\documentclass{article}

\usepackage{amsmath}
\usepackage{mathrsfs}
\usepackage{amssymb}
\usepackage{amsthm}

\usepackage{geometry}
\newtheorem{lemma}{Lemma}[section]
\newtheorem{theorem}[lemma]{Theorem}

\newtheorem{proposition}[lemma]{Proposition}
\newtheorem{definition}[lemma]{Definition}

\newtheorem{example}[lemma]{Example}
\newtheorem{remark}[lemma]{Remark}

\def\1{{\bf 1}}
\newcommand{\id}{{\rm id}}

\newcommand{\ito}[3]{\left(^{\ \ {#3}}_{{#1}\ {#2}}\right)}

\begin{document}
	\title{Uniqueness of vertex operator algebras
		arising from GKO-construction}
		\date{January 11, 2026}
	\maketitle
	
	\begin{center}
		Yuhan Gu\\
Email address: gyh.guyuhan@outlook.com\\
School of Mathematics and Statistics, 
		Qingdao University
\end{center}

\begin{center}
		Wen Zheng\footnote{Corresponding author. The author is supported by the NSFC No.~12201334 and the National Science Foundation of Shandong Province No.~ZR2022QA023.}\\
Email address: wzheng14@qdu.edu.cn\\
Qingdao International Mathematics Center
	\\School of Mathematics and Statistics, Qingdao University
	\end{center}

	\begin{abstract}
		A series of vertex operator algebras are constructed 
		by 
		GKO-construction, which is a generalization of 3A-algebra and 6A-algebra. 
		It is proved their vertex operator algebra structures are unique under nonzero assumptions on some elements of braiding matrices.
		Furthermore, we show each of them is generated by weight two subspace, i.e. the Griess algebra.
	\end{abstract}
	
	\setcounter{tocdepth}{2}
	\tableofcontents
	
	\section{Introduction}
	Constructing a series of vertex operator algebras (VOAs) $\mathcal{U}_k$ by GKO construction,
	and proving that $\mathcal{U}_k$ has a unique VOA structure
	are the main missions of this paper. 
	We briefly chart the history about structures and classifications of VOA
	before explaining any detailed methodology.
	
	Vertex (operator) algebra was introduced and studied
	extensively in mathematics 
	since 1980s \cite{B,FLM,FHL}. 
	Its properties show analogies of 
	commutative algebras, associative algebras,
	differential algebras and graded algebras.
	With these enriched structures,
	vertex operator algebras have built 
	deep connection between various aspects such as 
	group theories, modular forms, Lie algebras, 
	conformal nets, 
	tensor categories, and mathematical physics, etc,
	making it one of the most colorful field of study.
	
	People have already known a lot of examples of VOAs,
	such as Virasoro VOAs, Heisenberg VOAs, Affine VOAs,
	and Lattice VOAs. Apart from exsiting examples,
	new VOAs can be constructed by extensions, coset constructions,
	and orbifold theories.
	
	One of the most important VOAs 
	is the moonshine VOA $V^{\natural}$
	arising from the study of moonshine conjecture.
	It is proved in \cite{DMZ} that $V^{\natural}$ contains a 
	subalgebra $L(1/2,0)^{\otimes 48}$,
	where $L(1/2,0)$ is the simple Virasoro VOA with central charge $1/2$.
	
	Virasoro VOAs, especially unitary series of Virasoro VOAs, are the simplest and well-studied VOAs, which
	have many ideal properties that clarify the structure of them.
	Especially, we use GKO-construction 
	to get a series of VOAs $\mathcal{U}_k$, 
	and the fusion rules and braiding matrices \cite{FFK,H1,H2} of unitary series of Virasoro VOAs
	paly a vital role in our computations.
	
	One of our motivations is to generalize previous results about 3A-algebra \cite{M,SY} as well as 6A-algebra \cite{DJY,JZ}. For example, the $k=0$ and $k=1$ cases of VOAs $\mathcal{U}_k$ we constructed correspond  3A-algebra and 6A-algebra respectively. See Section 3 for the details. Another motivation is to try to understand Conjecture A.14 in \cite{LY}, which says the VOA generated by several Ising vectors related to 3A-algebra and 2A-algebra has unique VOA structure.
	In fact, VOA structures generated by two Ising vectors have been extensively studied \cite{M,SY,LYY, DJY, DZ, JZ, Zheng}. We would like to use similar methods like braiding matrices and fusion rules of Virasoro VOAs in \cite{DJY,DZ} to prove the uniqueness of VOA structure of $\mathcal{U}_k$, but the difficult part is to prove the uniqueness of $\mathcal{U}_k$ for all $k$ simultaneously.
	We give this proof under the nonvanishing assumption of some elements of the braiding matrices. For explicit computations of the elements of braiding matrices of $\mathcal{U}_1$ and $\mathcal{U}_2$, See Example 5.9 and 5.10.
	 
	
	
	Now we give the sketch of our work. 
	By directly applying GKO-construction,
	we can decompose 
	$$\left(\mathcal{L}\left(3,0\right)\otimes\mathcal{L}\left(1,0\right)\otimes\mathcal{L}\left(1,0\right)\oplus\mathcal{L}\left(3,3\right)\otimes\mathcal{L}\left(1,1\right)\otimes\mathcal{L}\left(1,0\right)\right)\otimes\mathcal{L}(1,0)^{\otimes k}$$
	(See Section 2.5 for details) into
	$\bigoplus_{i}\mathcal{U}_{k,i}\otimes \mathcal{L}(k+5,i)$.
	$\mathcal{U}:=\mathcal{U}_k:=\mathcal{U}_{k,0}$ is defined by taking the commutant of $\mathcal{L}(k+5,0)$ in GKO-construction.
	%
	More precisely, we have the decomposition
	$\mathcal{U}_k=\bigoplus U^i=\bigoplus P^i\otimes Q^i$,
	where $P_i$ are $\mathcal{U}_{k-1}$ modules and $Q_i$ are modules of Virasoro VOA. 
	Thus possible vertex operator is determined by fusion rules between those modules.
	We proved that fusion rules for $Q_i$ are enough for us.
	Since fusion rules for $Q_i$ are either $1$ or $0$, 
	we can put $Y=\sum \lambda_{i,j}^k \mathcal{Y}_{i,j}^k$,
	where $\mathcal{Y}_{i,j}^k$ is basis of intertwining operator of type 
	$\ito{U^i}{U^j}{U^k}$.
	Using mirror extension theory \cite{Lin} and the associativity of the extension, we prove $\lambda_{i,j}^k\neq 0$ for all $i,j,k$ compatible with fusion rules.
	For a fixed $Y$ we may therefore choose the bases so that $\lambda_{i,j}^k=1$, i.e. $Y=\sum \mathcal{Y}_{i,j}^k$. For another possible VOA structure $\overline{Y}=\sum \lambda_{i,j}^k \mathcal{Y}_{i,j}^k$,
	we need to show $(\mathcal{U},Y)\cong(\mathcal{U},\overline{Y})$.
	Under the nonvanishing assumption of some elements of the braiding matrices,
	we can assert $\left(\lambda_{i,j}^k\right)^2=1$ if $N_{i,j}^k=1$.
	Then we use $\lambda_{i,j}^k$ to define a map $\sigma$ between $(\mathcal{U},Y)$ and $(\mathcal{U},\overline{Y})$,
	and prove it is exactly the isomorphism we want.
	Furthermore, we show that $\mathcal{U}$ is generated by Griess algebra.
	
	This paper is organized as follows: 
	In Section 2, we recall some basic facts about
	vertex operator algebras and related topics needed in this paper. 
	In Section 3, the module $\mathcal{U}_k$ is constructed.
	In Section 4, we prove $\lambda_{i,j}^k\neq 0$ for all $i,j,k$ compatible with fusion rules.
	In Section 5, the isomorphism $\sigma$ is defined, and our main goal, the uniqueness of VOA structure on $\mathcal{U}_k$, is proved.
	In Section 6, we prove that $\mathcal{U}_k$ are generated by Griess algebra.
	Appendix~\ref{app:elementwise} gives an elementwise proof of the nonvanishing result using four-point functions.
	
	\section{Preliminary}
	In this section, we review the basics on vertex operator algebras,
	the theory of quantum dimensions from \cite{DJX}, the coset realization
	of the discrete series of the unitary representations for the Virasoro
	algebra \cite{GKO}, the braiding matrices for certain Virasoro
	vertex operator algebras \cite{FFK}, 
	the structure of $6A$-algebra \cite{DJY},
	and mirror extensions for rational vertex operator algebras \cite{Lin}.
	
	\subsection{Basics}
	Let $V=(V,\ Y,\ \mathbf{1},\ \omega)$ be a vertex operator algebra.
	Let $Y(v,\ z)=\sum_{n\in\mathbb{Z}}v_{n}z^{-n-1}$ denote the vertex
	operator of $V$ for $v\in V$, where $v_{n}\in\mbox{End}(V)$. We
	first recall some basic notions from \cite{FLM,Z,DLM1,DLM3}.
	
	\begin{definition} $V_{2}$ is called \emph{Griess algebra}. A vector $v\in V_{2}$ is called a \emph{Virasoro
			vector with central charge $c_{v}$ }if it satisfies\emph{ $v_{1}v=2v$
		}and $v_{3}v=\frac{c_{v}}{2}\mathbf{1}$. Then the operators $L_{n}^{v}:=v_{n+1},\ n\in\mathbb{Z}$,
		satisfy the Virasoro commutation relation
		\[
		\left[L_{m}^{v},\ L_{n}^{v}\right]=\left(m-n\right)L_{m+n}^{v}+\delta_{m+n,\ 0}\frac{m^{3}-m}{12}c_{v}
		\]
		for $m,\ n\in\mathbb{Z}.$ A Virasoro vector $v\in V_{2}$ with central
		charge $1/2$ is called an \emph{Ising vector }if $v$ generates the
		Virasoro vertex operator algebra $L(1/2,\ 0)$.
		
	\end{definition}
	
	\begin{definition} An \emph{automorphism} $g$ of a vertex operator
		algebra $V$ is a linear isomorphism of $V$ satisfying $g(\bf{1})=\bf{1}$, $g\left(\omega\right)=\omega$
		and $gY\left(v,z\right)g^{-1}=Y\left(gv,z\right)$ for any $v\in V$.
		We denote by $\mbox{Aut}\left(V\right)$ the group of all automorphisms
		of $V$. \end{definition}
	
	For a subgroup $G\le\mbox{Aut}\left(V\right)$ the fixed point set
	$V^{G}=\left\{ v\in V|g\left(v\right)=v,\forall g\in G\right\} $
	has a vertex operator algebra structure.
	
	Let $g$ be an automorphism of a vertex operator algebra $V$ of order
	$T$. Denote the decomposition of $V$ into eigenspaces of $g$ as:
	
	\[
	V=\oplus_{r\in\mathbb{Z}/T\text{\ensuremath{\mathbb{Z}}}}V^{r}
	\]
	where $V^{r}=\left\{ v\in V|gv=e^{2\pi ir/T}v\right\} $.
	
	\begin{definition} A \emph{weak $g$-twisted $V$-module} $M$ is
		a vector space with a linear map
		\begin{align*}
			Y_{M}: & V\to\left(\text{End}M\right)\{z\}\\
			& v\mapsto Y_{M}\left(v,z\right)=\sum_{n\in\mathbb{Q}}v_{n}z^{-n-1}\ \left(v_{n}\in\mbox{End}M\right)
		\end{align*}
		which satisfies the following: for all $0\le r\le T-1$, $u\in V^{r}$,
		$v\in V$, $w\in M$,
		
		\[
		Y_{M}\left(u,z\right)=\sum_{n\in-\frac{r}{T}+\mathbb{Z}}u_{n}z^{-n-1},
		\]
		
		\[
		u_{l}w=0\ for\ l\gg0,
		\]
		
		\[
		Y_{M}\left(\mathbf{1},z\right)=Id_{M},
		\]
		
		\[
		z_{0}^{-1}\text{\ensuremath{\delta}}\left(\frac{z_{1}-z_{2}}{z_{0}}\right)Y_{M}\left(u,z_{1}\right)Y_{M}\left(v,z_{2}\right)-z_{0}^{-1}\delta\left(\frac{z_{2}-z_{1}}{-z_{0}}\right)Y_{M}\left(v,z_{2}\right)Y_{M}\left(u,z_{1}\right)
		\]
		
		\begin{equation}
			=z_{2}^{-1}\left(\frac{z_{1}-z_{0}}{z_{2}}\right)^{-r/T}\delta\left(\frac{z_{1}-z_{0}}{z_{2}}\right)Y_{M}\left(Y\left(u,z_{0}\right)v,z_{2}\right),\label{Jacobi for twisted V-module}
		\end{equation}
		where $\delta\left(z\right)=\sum_{n\in\mathbb{Z}}z^{n}$.
		
	\end{definition}
	
	\begin{definition}
		
		A $g$-\emph{twisted $V$-module} is a weak $g$-twisted $V$-module
		$M$ which carries a $\mathbb{C}$-grading induced by the spectrum
		of $L(0)$ where $L(0)$ is the component operator of $Y(\omega,z)=\sum_{n\in\mathbb{Z}}L(n)z^{-n-2}.$
		That is, we have $M=\bigoplus_{\lambda\in\mathbb{C}}M_{\lambda},$
		where $M_{\lambda}=\left\{ w\in M|L(0)w=\lambda w\right\} $. Moreover,
		$\dim M_{\lambda}$ is finite and for fixed $\lambda,$ $M_{\frac{n}{T}+\lambda}=0$
		for all small enough integers $n.$
		
	\end{definition}
	
	\begin{definition}An \emph{admissible $g$-twisted $V$-module} $M=\oplus_{n\in\frac{1}{T}\mathbb{Z}_{+}}M\left(n\right)$
		is a $\frac{1}{T}\mathbb{Z}_{+}$-graded weak $g$-twisted module
		such that $u_{m}M\left(n\right)\subset M\left(\mbox{wt}u-m-1+n\right)$
		for homogeneous $u\in V$ and $m,n\in\frac{1}{T}\mathbb{Z}.$
		
	\end{definition}
	
	If $g=Id_{V}$ we have the notions of weak, ordinary and admissible
	$V$-modules \cite{DLM2}.
	
	\begin{definition}A vertex operator algebra $V$ is called \emph{$g$-rational}
		if the admissible $g$-twisted module category is semisimple. $V$
		is called \emph{rational} if $V$ is $1$-rational. \end{definition}
	
	It was proved in \cite{DLM2} that if $V$ is rational then there
	are only finitely irreducible admissible $V$-modules up to isomorphism
	and each irreducible admissible $V$-module is ordinary. Let $M^{0},M^{1},$
	$\cdots,M^{d}$ be all the irreducible modules up to isomorphism with
	$M^{0}=V$. Then there exists $h_{i}\in\mathbb{C}$ for $i=0,\cdots,d$
	such that
	\[
	M^{i}=\oplus_{n=0}^{\infty}M_{h_{i}+n}^{i}
	\]
	where $M_{h_{i}}^{i}\not=0$ and $L\left(0\right)|_{M_{h_{i}+n}^{i}}=h_{i}+n$,
	$\forall n\in\mathbb{Z}_{+}$. $h_{i}$ is called the \emph{conformal
		weight} of $M^{i}$. We denote $M^{i}\left(n\right)=M_{h_{i}+n}^{i}.$
	
	Let $M=\oplus_{\lambda_{\in\mathbb{C}}}M_{\lambda}$ be a $V$-module.
	The restricted dual of $M$ is defined by $M'=\oplus_{\lambda_{\in\mathbb{C}}}M_{\lambda}^{\ast}$
	where $M_{\lambda}^{\ast}=\text{Hom}_{\mathbb{C}}\left(M_{\lambda},\mathbb{C}\right).$
	It was proved in \cite{FHL} that $M'=\left(M',Y_{M'}\right)$ is
	naturally a $V$-module such that
	\[
	\left\langle Y_{M'}\left(v,z\right)f,u\right\rangle =\left\langle f,Y_{M}\left(e^{zL\left(1\right)}\left(-z^{-2}\right)^{L\left(0\right)}v,z^{-1}\right)u\right\rangle ,
	\]
	for $v\in V,$ $f\in M'$ and $u\in M$, and $\left(M'\right)'\cong M$.
	Moreover, if $M$ is irreducible, so is $M'$. A $V$-module $M$
	is said to be \emph{self dual} if $M\cong M'$.
	
	\begin{definition} A vertex operator algebra $V$ is said to be \emph{$C_{2}$-cofinite}
		if $V/C_{2}(V)$ is finite dimensional, where $C_{2}(V)=\langle v_{-2}u|v,u\in V\rangle.$
	\end{definition}
	
	\begin{definition}A vertex operator algebra $V=\oplus_{n\in\mathbb{Z}}V_{n}$
		is said to be of \emph{CFT type} if $V_{n}=0$ for negative $n$ and
		$V_{0}=\mathbb{C}1$. \end{definition}
	
	\begin{definition} Let $\left(V,Y\right)$ be a vertex operator algebra
		and let $\left(M^{i},Y^{i}\right),\ \left(M^{j},Y^{j}\right)$ and
		$\left(M^{k},Y^{k}\right)$ be three $V$-modules. An \emph{intertwining
			operator of type $\left(\begin{array}{c}
				M^{k}\\
				M^{i}\ M^{j}
			\end{array}\right)$} is a linear map
		\begin{gather*}
			\mathcal{Y}\left(\cdot,z\right):\ M^{i}\to\text{\ensuremath{\mbox{Hom}\left(M^{j},\ M^{k}\right)\left\{ z\right\} }}\\
			u\mapsto\mathcal{Y}\left(u,z\right)=\sum_{n\in\mathbb{Q}}u_{n}z^{-n-1}
		\end{gather*}
		satisfying:
		
		(1) for any $u\in M^{i}$ and $v\in M^{j}$, $u_{n}v=0$ for $n$
		sufficiently large;
		
		(2) $\mathcal{Y}(L_{-1}v,\ z)=\left(\frac{d}{dz}\right)\mathcal{Y}\left(v,z\right)$
		for $v\in M^{i}$;
		
		(3) (Jacobi Identity) for any $u\in V,\ v\in M^{i}$,
		\begin{alignat*}{1}
			& z_{0}^{-1}\delta\left(\frac{z_{1}-z_{2}}{z_{0}}\right)Y^{k}\left(u,z_{1}\right)\mathcal{Y}\left(v,z_{2}\right)-z_{0}^{-1}\delta\left(\frac{-z_{2}+z_{1}}{z_{0}}\right)\mathcal{Y}\left(v,z_{2}\right)Y^{j}\left(u,z_{1}\right)\\
			& =z_{2}^{-1}\delta\left(\frac{z_{1}-z_{0}}{z_{2}}\right)\mathcal{Y}\left(Y^{i}\left(u,z_{0}\right)v,z_{2}\right).
		\end{alignat*}
		The space of all intertwining operators of type $\left(\begin{array}{c}
			M^{k}\\
			M^{i}\ M^{j}
		\end{array}\right)$ is denoted $I_{V}\left(\begin{array}{c}
			M^{k}\\
			M^{i}\ M^{j}
		\end{array}\right)$. Without confusion, we also denote it by $I_{i,j}^{k}.$ Let $N_{i,\ j}^{k}=\dim I_{i,j}^{k}$.
		These integers $N_{i,j}^{k}$ are called the \emph{fusion rules}.
	\end{definition}
	
	
	
	Let $V^{1}$ and $V^{2}$ be vertex operator algebras. Let $M^{i}$
	, $i=1,2,3$, be $V^{1}$-modules, and $N^{i}$, $i=1,2,3$, be $V^{2}$-modules.
	Then $M^{i}\otimes N^{i}$, $i=1,2,3$, are $V^{1}\otimes V^{2}$-modules
	by \cite{FHL}. The following property was given in \cite{ADL}:
	
	\begin{proposition} \label{fusion of tensor product}If $N_{M^{1},M^{2}}^{M^{3}}<\infty$
		or $N_{N^{1},N^{2}}^{N^{3}}<\infty,$ then
		\[
		N_{M^{1}\otimes N^{1},M^{2}\otimes N^{2}}^{M^{3}\otimes N^{3}}=N_{M^{1},M^{2}}^{M^{3}}N_{N^{1},N^{2}}^{N^{3}}.
		\]
	\end{proposition}
	
	Let $M^{1}$ and $M^{2}$ be $V$-modules. A fusion product for the
	ordered pair $\left(M^{1},M^{2}\right)$ is a pair $\left(M,F\left(\cdot,z\right)\right)$
	which consists of a $V$-module $M$ and an intertwining operator
	$F\left(\cdot,z\right)$ of type $\left(_{M^{1},M^{2}}^{\ \ M}\right)$
	such that the following universal property holds: For any $V$-module
	$W$ and any intertwining operator $I\left(\cdot,z\right)$ of type
	$\left(_{M^{1},M^{2}}^{\ \ W}\right)$, there exists a unique $V$-homomorphism
	$\phi$ from $M$ to $W$ such that $I\left(\cdot,z\right)=\phi\circ F\left(\cdot,z\right).$
	It is clear from the definition that if a tensor product of $M^{1}$
	and $M^{2}$ exists, it is unique up to isomorphism. In this case,
	we denote the \emph{fusion product} by $M^{1}\boxtimes_{V}M^{2}.$
	
	The basic result is that the fusion product exists if $V$ is rational.
	Let $M,N$ be irreducible $V$ -modules, we shall often consider the
	fusion product
	\[
	M\boxtimes_{V}N=\sum_{W}N_{M,\ N}^{W}W
	\]
	where $W$ runs over the set of equivalence classes of irreducible
	$V$-modules.
	
	\begin{definition} Let $V$ be a simple vertex operator algebra.
		A simple $V$-module $M$ is called \emph{simple current} if for
		any irreducible $V$-module $W$, $M\boxtimes_{V}W$ exists and is
		also a simple $V$-module. \end{definition}
	
	The following proposition is from \cite{FHL}:
	
	\begin{proposition} \label{extension property} Let $V$ be a vertex
		operator algebra and $V'$ be its restricted dual. For $u,v,w\in V$
		and $t\in V'$, we have the following equality of rational functions
		
		\begin{equation}
			\iota_{12}^{-1}\left\langle t,Y\left(u,z_{1}\right)Y\left(v,z_{2}\right)w\right\rangle =\iota_{21}^{-1}\left\langle t,Y\left(v,z_{2}\right)Y\left(u,z_{1}\right)w\right\rangle \label{communitivity}
		\end{equation}
		
		\begin{equation}
			\iota_{12}^{-1}\left\langle t,Y\left(u,z_{1}\right)Y\left(v,z_{2}\right)w\right\rangle =\left(\iota_{20}^{-1}\left\langle t,Y\left(Y\left(u,z_{0}\right)v,z_{2}\right)w\right\rangle \right)|_{z_{0}=z_{1}-z_{2}}\label{associavity}
		\end{equation}
		where $\iota_{12}^{-1}f\left(z_{1},z_{2}\right)$ denotes the formal
		power expansion of an analytic function $f\left(z_{1},z_{2}\right)$
		in the domain $\left|z_{1}\right|>\left|z_{2}\right|$ . \end{proposition}
	
	The following result about bilinear form on $V$ is from \cite{Li}:
	
	\begin{theorem}\label{bilinear form}The space of invariant bilinear
		forms on $V$ is isomorphic to the space
		\[
		\left(V_{0}/L\left(1\right)V_{1}\right)^{*}=\mbox{Hom}_{\mathbb{C}}\left(V_{0}/L\left(1\right)V_{1},\mathbb{C}\right).
		\]
	\end{theorem}
	
	\subsection{Quantum Galois theory}
	
	Now we recall quantum Galois theory and quantum dimensions from \cite{DM}
	and \cite{DJX}. Let $V$ be a simple vertex operator algebra and
	$G$ a finite and faithful group of automorphisms of $V$. Let $\text{Irr}\left(G\right)$
	be the set of simple characters $\chi$ of $G$. As $\mathbb{C}G$-module,
	each homogeneous space $V_{n}$ of $V$ is finite dimensional, and
	$V$ can be decomposed into a direct sum of graded subspaces
	\[
	V=\oplus_{\chi\in\text{Irr}\left(G\right)}V^{\chi},
	\]
	where $V^{\chi}$ is the subspace of $V$ on which $G$ acts according
	to the character $\chi$. The following theorem is from \cite{DM,DLM4}. 
	
	\begin{theorem} \label{classical galois theory}Suppose that $V$
		is a simple vertex operator algebra and that $G$ is a finite and
		faithful solvable group of automorphisms of $V$. Then the following
		hold:
		
		(i) Each $V^{\chi}$ is nonzero;
		
		(ii) For $\chi\in\text{Irr}\left(G\right)$, each $V^{\chi}$ is a
		simple module for the $G$-graded vertex operator algebra $\mathbb{C}G\otimes V^{G}$
		of the form
		
		\[
		V^{\chi}=M_{\chi}\otimes V_{\chi},
		\]
		where $M_{\chi}$ is the simple $\mathbb{C}G$-module affording $\chi$
		and where $V_{\chi}$ is a simple $V^{G}$-module.
		
		(iii) The map $M_{\chi}\mapsto V_{\chi}$ is a bijection from the
		set of simple $\mathbb{C}G$-modules to the set of (inequivalent)
		simple $V^{G}$-modules which are contained in $V$.
		
	\end{theorem}
	
	Now we recall the notion of quantum dimension from \cite{DJX}. Let
	$V$ be a vertex operator algebra of CFT type and $M$ a $V$-module,
	the formal character of $M=\oplus_{n\in\mathbb{Z}_{+}}M_{\lambda+n}$
	is defined to be
	\[
	Ch_{q}M=\mbox{tr}_{q}M=\text{tr}q^{L(0)-c/24}=q^{\lambda-c/24}\sum_{n\in\mathbb{Z}_{+}}(\dim M_{\lambda+n})q^{n}
	\]
	where $\lambda$ is the conformal weight of $M$. The quantum dimension
	of $M$ over $V$ is defined as:
	\[
	q\dim_{V}M=\lim_{q\to1^{-}}\frac{Ch_{q}M}{Ch_{q}V}.
	\]
	The following result is from Theorem 6.3 in \cite{DJX}:
	
	\begin{theorem} \label{quantum dimension and orbifold module}Let
		$V$ be a rational and $C_{2}$-cofinite simple vertex operator algebra.
		Assume $V$ is $g$-rational and the conformal weight of any irreducible
		$g$-twisted $V$-module is positive except for $V$ itself for all
		$g\in G$. Then
		\[
		q\dim_{V^{G}}V_{\chi}=\dim W_{\chi}.
		\]
	\end{theorem}
	
	For convenience, from now on, we say a vertex operator algebra $V$
	is ``good'' if it satisfies the following conditions: $V$ is a
	rational and $C_{2}$-cofinite simple vertex operator algebra of CFT
	type with $V\cong V'$. Let $M^{0},\ M^{1},\ \cdots,\ M^{d}$ be all
	the inequivalent irreducible $V$-modules with $M^{0}\cong V$ . The
	corresponding conformal weights $\lambda_{i}$ satisfy $\lambda_{i}>0$
	for $0<i\le d$.
	
	The following properties of quantum dimensions are from \cite{DJX}
	:
	
	\begin{proposition} \label{product property of quantum dimension }\label{quantum dimension and simple current}
		Let $V$ be a ``good'' vertex operator algebra. Then
		
		(i) $q\dim_{V}\left(M^{i}\boxtimes M^{j}\right)=q\dim_{V}M^{i}\cdot q\dim_{V}M^{j},$
		$\forall i,j$.
		
		(ii) A $V$-module $M^{i}$ is a simple current if and only if $q\dim_{V}M^{i}=1$.
		
		(iii) $q\dim_{V}M^{i}\in\left\{ 2\cos\left(\pi/n\right)|n\ge3\right\} \cup\left\{ a|2\le a<\infty,a\ is\ algebraic\right\} .$
		
	\end{proposition}
	
	\begin{definition} Let $V$ be a vertex operator algebra with finitely
		many inequivalent irreducible modules $M^{0},\cdots,M^{d}$. The \emph{global
			dimension} of $V$ is defined as
		\[
		\text{glob}\left(V\right)=\sum_{i=0}^{d}\left(q\dim_{V}M^{i}\right)^{2}.
		\]
	\end{definition}
	
	\begin{remark} \label{product property of global dimension}Let $U$
		and $V$ be ``good'' vertex operator algebras, $M$ be a $U$-module
		and $N$ be a $V$-module. Then Lemma 2.10 of \rm{\cite{ADJR}} gives
		
		\[
		q\dim_{U\otimes V}M\otimes N=q\dim_{U}M\cdot q\dim_{V}N,
		\]
		\[
		\text{glob}\left(U\otimes V\right)=\text{glob}\left(U\right)\cdot\text{glob}\left(V\right).
		\]
	\end{remark}
	
	Let $V$ be a vertex operator algebra, recall that a simple vertex
	operator algebra containing $V$ is called an \emph{extension} $U$
	of $V$. Now we have the following theorem \cite{ABD,HKL,ADJR}:
	
	\begin{theorem} \label{rationality of extesnion of VOA } \label{global dim of VOA and subVOA}
		Let $V$ be a ``good'' vertex operator algebra. Let $U$ be a simple
		vertex operator algebra which is an extension of $V$. Then $U$ is
		also ``good'' and
		\[
		\text{glob}\left(V\right)=\text{glob}\left(U\right)\cdot\left(q\dim_{V}\left(U\right)\right)^{2}.
		\]
		
	\end{theorem}
	
	\subsection{\label{subsec:The-unitary-series}The unitary series of the Virasoro
		VOAs}
	
	Now we recall notations about unitary minimal models of Virasoro algebra
	from \cite{FFK}. The models are parameterized by a complex number
	$\alpha_{-}^{2}$, related to the central charge of the Virasoro algebra
	by $c=13-6\alpha_{-}^{2}-6\alpha_{-}^{-2}$ where $\alpha_{-}^{2}=\frac{p}{p'}$
	and $\left|p-p'\right|=1.$ Without loss of generality, we write $p'=p+1$
	and denote $c_{p}=1-\frac{6}{p\left(p+1\right)}$ with $p=2,3,4,\cdots$.
	The label $I$ stands for a pair $\left(i',i\right)$ of positive
	integers and the corresponding highest weight is
	\begin{equation}
		h_{I}=h_{\left(i'\ i\right)}^{\left(p\right)}=\frac{1}{4}\left(i'^{2}-1\right)\alpha_{-}^{2}-\frac{1}{2}\left(i'i-1\right)+
		\frac{1}{4}\left(i^{2}-1\right)\alpha_{-}^{-2}=\frac{\left(pi'-\left(p+1\right)i\right)^{2}-1}{4p\left(p+1\right)}.\label{minimal model notation}
	\end{equation}
	for $1\le i'\le p,$ $1\le i\le p-1.$ We denote such unitary minimal
	models of Virasoro algebra by $L\left(c_{p},h_{\left(i',i\right)}^{\left(p\right)}\right).$
	
	\begin{remark} \label{pairs corr. to U's weights for unitary model}Use
		the above notation, we see that the central charge of the model $L\left(\frac{25}{28},0\right)$
		corresponds to the parameter $\alpha_{-}^{2}=\frac{7}{8}$ with $p=7,$
		$p'=8$. The highest weights for irreducible $L\left(\frac{25}{28},0\right)$-modules
		are
		
		\begin{equation}
			\left\{ 0,\frac{5}{32},\frac{3}{4},\frac{57}{32},\frac{13}{4},\frac{165}{32},\frac{15}{2},\frac{5}{14},\frac{3}{224},
			\frac{3}{28},\frac{143}{224},\frac{45}{28},\frac{675}{224},\frac{34}{7},\frac{9}{7},\frac{99}{224},\frac{1}{28},\frac{15}{224},\frac{25}{28},\frac{323}{224},
			\frac{39}{14}\right\} .\label{highest weights for 25/28}
		\end{equation}
		In particular, the pairs $\left(1,1\right)$, $\left(1,5\right)$
		and $\left(1,3\right)$ correspond to the highest weights $0$, $\frac{34}{7}$
		and $\frac{9}{7}$ respectively. \end{remark}
	
	Also note that the fusion rules for irreducible $L\left(c_{p},0\right)$-modules
	are as follows \cite{W}:
	
	\begin{definition} An ordered triple of pairs of integers $\left(\left(i',i\right),\left(j',j\right),\left(k',k\right)\right)$
		is called \emph{admissible }if $1\le i',j',k'\le p+1,1\le i,j,k\le p$,
		$i'+j'+k'<2\left(p+1\right),$ $i+j+k<2p$, $i'<j'+k'$, $j'<i'+k'$,
		$k'<i'+j'$, $i<j+k$, $j<i+k$, $k<i+j$, and the sums $i'+j'+k'$,
		$i+j+k$ are odd. \end{definition}
	
	\begin{proposition} \label{fusion rules of virasoro modules}The
		fusion rules between $L\left(c_{p},0\right)$-modules $L\left(c_{p},h_{\left(i',i\right)}^{\left(p\right)}\right),$
		$L\left(c_{p},h_{\left(j',j\right)}^{\left(p\right)}\right)$ are
		\[
		L\left(c_{p},h_{\left(i',i\right)}^{\left(p\right)}\right)\boxtimes L\left(c_{p},h_{\left(j',j\right)}^{\left(p\right)}\right)=
		\sum_{\left(k',k\right)}N_{\left(i',i\right),\left(j',j\right)}^{\left(k',k\right)}L\left(c_{p},h_{\left(k',k\right)}^{\left(p\right)}\right),
		\]
		where $N_{\left(i',i\right),\left(j',j\right)}^{\left(k',k\right)}$
		is $1$ iff $\left(\left(i',i\right),\left(j',j\right),\left(k',k\right)\right)$
		is an admissible triple of pairs and $0$ otherwise. \end{proposition}
	
	\subsection{\label{subsec:Braiding-matrix}Braiding matrices}
	
	Now we recall four point functions. Let $V$ be a rational and $C_{2}$-cofinite
	vertex operator algebra of CFT type and $V\cong V'$. Let $M^{a_{1}},M^{a_{2}},M^{a_{3}},M^{a_{4}}$
	be four irreducible $V$-modules. By Lemma 4.1 in \cite{H1} we know
	that for $u_{a_{i}}\in M^{a_{i}},$
	
	\[
	\left\langle u_{a_{4}'},\mathcal{Y}_{a_{1},a_{5}}^{a_{4}}\left(u_{a_{1}},z_{1}\right)\mathcal{Y}_{a_{2},a_{3}}^{a_{5}}
	\left(u_{a_{2}},z_{2}\right)u_{a_{3}}\right\rangle ,
	\]
	
	\[
	\left\langle u_{a_{4}'},\mathcal{Y}_{a_{2},a_{6}}^{a_{4}}\left(u_{a_{1}},z_{2}\right)\mathcal{Y}_{a_{1},a_{3}}^{a_{6}}\left(u_{a_{1}},z_{1}\right)u_{a_{3}}\right\rangle
	\]
	are analytic on $\left|z_{1}\right|>\left|z_{2}\right|>0$ and $\left|z_{2}\right|>\left|z_{1}\right|>0$
	respectively, and can both be analytically extended to multi-valued
	analytic functions on
	\[
	R=\left\{ \left(z_{1},z_{2}\right)\in\mathbb{C}^{2}|z_{1},z_{2}\not=0,z_{1}\not=z_{2}\right\} .
	\]
	One can lift the multi-valued functions on $R$ to single-valued functions
	on the universal covering $\tilde{R}$ to $R$ as in \cite{H1}. We
	use
	
	\[
	E\left\langle u_{a_{4}'},\mathcal{Y}_{a_{1},a_{5}}^{a_{4}}\left(u_{a_{1}},z_{1}\right)\mathcal{Y}_{a_{2},a_{3}}^{a_{5}}\left(u_{a_{2}},z_{2}\right)u_{a_{3}}\right\rangle
	\]
	and
	\[
	E\left\langle u_{a_{4}'},\mathcal{Y}_{a_{2},a_{6}}^{a_{4}}\left(u_{a_{1}},z_{2}\right)\mathcal{Y}_{a_{1},a_{3}}^{a_{6}}\left(u_{a_{1}},z_{1}\right)u_{a_{3}}\right\rangle
	\]
	to denote those analytic functions.
	
	Let $\left\{ \mathcal{Y}_{a,b;i}^{c}|i=1,\cdots,N_{a,b}^{c}\right\} $
	be a basis of $I_{a,b}^{c}$. From \cite{H2},
	
	\[
	\left\{ E\left\langle u_{a_{4}'},\mathcal{Y}_{a_{1},a_{5};i}^{a_{4}}\left(u_{a_{1}},z_{1}\right)\mathcal{Y}_{a_{2},a_{3};j}^{a_{5}}\left(u_{a_{2}},z_{2}\right)u_{a_{3}}\right\rangle |i=1,\cdots,N_{a_{1},a_{5}}^{a_{4}},j=1,\cdots,N_{a_{2},a_{3}}^{a_{5}},\forall a_{5}\right\}
	\]
	is a linearly independent set. Fix a basis of intertwining operators.
	It was proved in \cite{KZ,TK} that
	\begin{alignat*}{1}
		& span\left\{ E\left\langle u_{a_{4}'},\mathcal{Y}_{a_{3},\mu;i}^{a_{4}}\left(u_{a_{3}},z_{1}\right)\mathcal{Y}_{a_{2},a_{1};j}^{\mu}\left(u_{a_{2}},z_{2}\right)u_{a_{1}}\right\rangle |i,j,\mu\right\} \\
		& =span\left\{ E\left\langle u_{a_{4}'},\mathcal{Y}_{a_{2},\gamma;k}^{a_{4}}\left(u_{a_{2}},z_{2}\right)\mathcal{Y}_{a_{3},a_{1};l}^{\gamma}\left(u_{a_{3}},z_{1}\right)u_{a_{1}}\right\rangle |k,l,\gamma\right\} ,
	\end{alignat*}
	where $u_{a_{i}}\in M^{a_{i}}$. Then there exists $\left(B_{a_{4,}a_{1}}^{a_{3},a_{2}}\right)_{\mu,\gamma}^{i,j;k,l}\in\mathbb{C}$
	such that
	\begin{alignat}{1}
		& E\left\langle u_{a_{4}'},\mathcal{Y}_{a_{3},\mu;i}^{a_{4}}\left(u_{a_{3}},z_{1}\right)\mathcal{Y}_{a_{2},a_{1};j}^{\mu}\left(u_{a_{2}},z_{2}\right)u_{a_{1}}\right\rangle \nonumber \\
		& =\sum_{k,l,\gamma}\left(B_{_{a_{4},a_{1}}}^{a_{3,}a_{2}}\right)_{\mu,\gamma}^{i,j;k,l}E\left\langle u_{a_{4}'},\mathcal{Y}_{a_{2},\gamma;k}^{a_{4}}\left(u_{a_{2}},z_{1}\right)\mathcal{Y}_{a_{3},a_{1};l}^{\gamma}\left(u_{a_{3}},z_{2}\right)u_{a_{1}}\right\rangle \label{braiding matrix property}
	\end{alignat}
	(see \cite{H1}). $B_{_{a_{4},a_{1}}}^{a_{3,}a_{2}}$ is called
	the\emph{ braiding matrix}.
	
	
	Recall that we have seen $\alpha_{-}^{2}=\frac{p}{p'}$ in Section
	\ref{subsec:The-unitary-series}. Now let $\alpha_{+}^{2}=-\frac{p'}{p}$,
	$x=\exp\left(2\pi i\alpha_{+}^{2}\right)$, $y=\exp\left(2\pi i\alpha_{-}^{2}\right)$,
	$\left[l\right]=x^{l/2}-x^{-l/2}$, $\left[l'\right]=y{}^{l'/2}-y{}^{-l'/2}.$
	Now we fix central charge $c_{p}$, denote $L\left(c_{p},h_{\left(i',i\right)}^{\left(p\right)}\right)$
	by $\left(i',i\right)$. Let $\left(a',a\right)$, $\left(m',m\right)$,
	$\left(n',n\right)$, $\left(c',c\right)$, $\left(b',b\right)$,
	$\left(d',d\right)$ be irreducible $L\left(c_{p},0\right)$-modules,
	the braiding matrices of screened vertex operators have the almost
	factorized form (cf. (2.19) of \cite{FFK}):
	\begin{alignat}{1}
		& \left(\tilde{B}_{\left(m',m\right),\left(n',n\right)}^{\left(a',a\right),\left(c',c\right)}\right){}_{\left(b',b\right),\left(d',d\right)}\nonumber \\
		& =i^{-\left(m'-1\right)\left(n-1\right)-\left(n'-1\right)\left(m-1\right)}\left(-1\right)^{1/2\left(a-b+c-d\right)\left(n'+m\right)+1/2\left(a'-b'+c'-d'\right)\left(n+m\right)}\label{FFK 2.19}\\
		& \cdot r\left(a',m',n',c'\right)_{b',d'}\cdot r\left(a,m,n,c\right)_{b,d},\nonumber
	\end{alignat}
	where the nonvanishing matrix elements of $r$-matrices are
	
	\begin{gather}
		r\left(a,1,n,c\right)_{a,c}=r\left(a,m,1,c\right)_{c,a}=1,\nonumber \\
		r\left(l\pm2,2,2,l\right)_{l\pm1,l\pm1}=x^{1/4},\nonumber \\
		r\left(l,2,2,l\right)_{l\pm1,l\pm1}=\mp x^{-1/4\mp l/2}\frac{\left[1\right]}{\left[l\right]},\nonumber \\
		r\left(l,2,2,l\right)_{l\pm1,l\mp1}=x^{-1/4}\frac{\left[l\pm1\right]}{\left[l\right]},\label{FFK 2.20}
	\end{gather}
	and the other $r$-matrices are given by the recursive relation
	\begin{gather}
		r\left(a,m+1,n,c\right)_{b,d}=\sum_{d_{1}\ge1}r\left(a,2,n,d_{1}\right)_{a_{1},d}\cdot r\left(a_{1},m,n,c\right)_{b,d_{1}},\nonumber \\
		r\left(a,m,n+1,c\right)_{b,d}=\sum_{d_{1}\ge1}r\left(a,m,2,c_{1}\right)_{b,d_{1}}\cdot r\left(d_{1},m,n,c\right)_{c_{1},d},\label{FFK 2.21}
	\end{gather}
	for any choice of $a_{1}$ and $c_{1}$ compatible with the fusion
	rules. The $r'$ matrices are given by the same formulas with the
	replacement $x\to x'$, $\left[\ \ \ \ \right]\to\left[\ \ \ \ \right]'.$
	
	\subsection{GKO construction of the unitary Virasoro VOA}
	
	Let $e,f$ and $h$ be the generators of $\mathfrak{sl}_{2}(\mathbb{C})$
	such that
	\[
	\left[e,f\right]=h,\ \left[h,e\right]=2e,\ \left[h,f\right]=-2f.
	\]
	Let $\left\langle \cdot,\cdot\right\rangle $ be the standard invariant
	bilinear form on $\mathfrak{sl}_{2}(\mathbb{C})$ defined by
	\[
	\left\langle h,h\right\rangle =2,\ \left\langle e,f\right\rangle =1,\ \left\langle e,e\right\rangle =\left\langle f,f\right\rangle =\left\langle h,e\right\rangle =\left\langle h,f\right\rangle =0.
	\]
	Let $\hat{\mathfrak{sl}}_{2}\left(\mathbb{C}\right)$ be the corresponding
	affine algebra of type $A_{1}^{(1)}$ and $\lambda_{0},\ \lambda_{1}$
	the fundamental weights for $\hat{\mathfrak{sl}}_{2}\left(\mathbb{C}\right)$.
	Denote
	\[
	\mathcal{L}(m,k)=\mathcal{L}\left(\left(m-k\right)\lambda_{0}+k\lambda_{1}\right)
	\]
	the irreducible highest weight module of $\hat{\mathfrak{sl}}_{2}\left(\mathbb{C}\right)$-module
	with highest weight $\left(m-k\right)\lambda_{0}+k\lambda_{1}$. It
	was proved in \cite{FZ} that $\mathcal{L}\left(m,0\right)$ has a
	natural vertex operator algebra structure for $m\in\mathbb{Z}_{+}$.
	The Virasoro vector $\omega^{m}$ of $\mathcal{L}\left(m,0\right)$
	is given by
	\[
	\omega^{m}=\frac{1}{2\left(m+2\right)}\left(\frac{1}{2}h_{-1}h+e_{-1}f+f_{-1}e\right)
	\]
	with central charge $\frac{3m}{m+2}$. Let $m\in\mathbb{Z}_{+}$,
	then $\mathcal{L}\left(m,0\right)$ is a rational vertex operator
	algebra and
	$$\left\{ \mathcal{L}\left(m,\ k\right)|k=0,1,\cdots,m\right\}$$
	is the set of all the irreducible 
	$\mathcal{L}\left(m,0\right)$-modules.
	Moreover, the fusion rules are given by
	
	\[
	\mathcal{L}\left(m,j\right)\boxtimes\mathcal{L}\left(m,k\right)=\sum_{i=\max\left\{ 0,j+k-m\right\} }^{\min\left\{ j,k\right\} }\mathcal{L}\left(m,j+k-2i\right).
	\]
	
	Let $\mathcal{L}\left(m,0\right)_{1}$ be the weight 1 subspace of
	$\mathcal{L}(m,0)$. Then $\mathcal{L}(m,0)_{1}$ has a structure
	of Lie algebra isomorphic to $\mathfrak{sl}_{2}(\mathbb{C})$ under
	$[a,b]=a_{0}b$, $\forall a,b\in\mathcal{L}(m,0)$. Let $h^{m}$,
	$e^{m}$, $f^{m}$ be the generators of $\mathfrak{sl}_{2}(\mathbb{C})$
	in $\mathcal{L}(m,0)_{1}$. Then $h^{m+1}:=h^{1}\otimes1+1\otimes h^{m},$
	$e^{m+1}:=e^{1}\otimes1+1\otimes e^{m}$ and $f^{m+1}:=f^{1}\otimes1+1\otimes f^{m}$
	generate a vertex operator subalgebra isomorphic to $\mathcal{L}\left(m+1,0\right)$
	in $\mathcal{L}\left(m,1\right)\otimes\mathcal{L}\left(m,0\right)$.
	It was proved in \cite{DL} and \cite{KR} that $\Omega^{m}:=\omega^{1}\otimes1+1\otimes\omega^{m}-\omega^{m+1}$
	also gives a Virasoro vector with central charge $c_{m+2}=1-6/\left(m+2\right)\left(m+3\right)$.
	Furthermore, $\omega^{m+1}$ and $\Omega^{m}$ are mutually commutative
	and $\Omega^{m}$ generates a simple Virasoro vertex operator algebra
	$L\left(c_{m+2},0\right)$. Therefore $\mathcal{L}\left(1,0\right)\otimes\mathcal{L}\left(m,0\right)$
	contains a vertex operator subalgebra isomorphic to $L\left(c_{m+2},0\right)\otimes\mathcal{L}\left(m+1,0\right)$.
	Note that both $L\left(c_{m+2},0\right)$ and $\mathcal{L}\left(m+1,0\right)$
	are rational and every $\mathcal{L}\left(1,0\right)\otimes\mathcal{L}\left(m,0\right)$-module
	can be decomposed into irreducible $L(c_{m+2},0)\otimes\mathcal{L}(m+1,0)$-submodules.
	We have the following decomposition \cite{GKO}:
	
	\begin{equation}
		\mathcal{L}\left(1,\epsilon\right)\otimes\mathcal{L}\left(m,n\right)=\bigoplus_{0\le s\le m+1,s\equiv n+\epsilon\ \rm{mod} 2}L\left(c_{m+2},h_{\left(s+1,n+1\right)}^{(m+2)}\right)\otimes\mathcal{L}\left(m+1,s\right)\label{eq:(2.3)}
	\end{equation}
	where $\epsilon=0,1$ and $0\le n\le m$. This is the GKO-construction
	of the unitary Virasoro vertex operator algebras.
	
	\subsection{Structure of 6A-algebra}
	$6A$-algebra $\mathcal{U}_{6A}$ is a
	vertex operator algebra of the Moonshine type generated by two Ising
	vectors whose inner product is $5/2^{10}.$
	
	A vector $v\in V_{2}$ is called a \emph{Virasoro
		vector with central charge $c_{v}$ }if it satisfies\emph{ $v_{1}v=2v$
	}and $v_{3}v=\frac{c_{v}}{2}\mathbf{1}$. Then the operators $L_{n}^{v}:=v_{n+1},\ n\in\mathbb{Z}$,
	satisfy the Virasoro commutation relation
	\[
	\left[L_{m}^{v},\ L_{n}^{v}\right]=\left(m-n\right)L_{m+n}^{v}+\delta_{m+n,\ 0}\frac{m^{3}-m}{12}c_{v}
	\]
	for $m,\ n\in\mathbb{Z}.$ A Virasoro vector $v\in V_{2}$ with central
	charge $1/2$ is called an \emph{Ising vector }if $v$ generates the
	Virasoro vertex operator algebra $L(1/2,\ 0)$.
	
	It is shown in \cite{C} that the structure of the subalgebra generated
	by two Ising vectors $e$ and $f$ in the algebra $V_{2}^{\natural}$
	depends on only the conjugacy class of $\tau_{e}\tau_{f}$, and the
	inner product $\left\langle e,f\right\rangle $ is given by the following
	table:
	\begin{center}
		\begin{tabular}{|c|c|c|c|c|c|c|c|c|c|}
			\hline
			$\left\langle \tau_{e}\tau_{f}\right\rangle ^{\mathbb{M}}$  & $1A$  & $2A$  & $3A$  & $4A$  & $5A$  & $6A$  & $3C$  & $4B$  & $2B$\tabularnewline
			\hline
			$\left\langle e,f\right\rangle $  & $1/4$  & $1/2^{5}$  & $13/2^{10}$  & $1/2^{7}$  & $3/2^{9}$  & $5/2^{10}$  & $1/2^{8}$  & $1/2^{8}$  & $0$\tabularnewline
			\hline
		\end{tabular}
		\par\end{center}

	Certain coset subalgebra of $V_{\sqrt{2}E_{8}}$ associated with extended
	$E_{8}$ diagram is constructed in \cite{LYY} by removing one node
	from the diagram. In each case, the coset subalgebra contains some
	Ising vectors and the coset subalgebra is generated by two Ising vectors
	with inner product the same as the number given in the table above.
	In particular, the coset subalgebra $\mathcal{U}_{6A}$ corresponding
	to the $6A$ case was constructed, i.e., the case with inner product
	$\frac{5}{2^{10}}.$ Let $\mathcal{V}$ be the $3A$-algebra, that
	is, the vertex operator algebra generated by two Ising vectors whose
	$\tau$-involutions generate $S_{3}$ and with inner product $\frac{13}{2^{10}}$.
	The candidates for $\mathcal{V}$ were given \cite{M} and it was
	proved in \cite{SY} that only one of these candidates actually exists
	and that there is unique vertex operator algebra structure on it.

	Now we recall the following results about the $3A$-algebra $\mathcal{V}$
	from \cite{SY} .
	
	\begin{lemma} \label{rationality of 3A}The $3A$-algebra $\mathcal{V}$
		is rational. \end{lemma}
	
	\begin{lemma} \label{U_3A modules} All the irreducible $\mathcal{V}$-modules
		are as follows:
		
		\begin{align*}
			\mathcal{V}=\mathcal{V}\left(0\right) & =\left(\left(L\left(\frac{4}{5},0\right)\oplus L\left(\frac{4}{5},3\right)\right)\otimes\left(L\left(\frac{6}{7},0\right)\oplus L\left(\frac{6}{7},5\right)\right)\right)\\
			& \oplus L\left(\frac{4}{5},\frac{2}{3}\right)^{+}\otimes L\left(\frac{6}{7},\frac{4}{3}\right)^{+}\oplus L\left(\frac{4}{5},\frac{2}{3}\right)^{-}\otimes L\left(\frac{6}{7},\frac{4}{3}\right)^{-},
		\end{align*}
		
		\begin{align*}
			\mathcal{V}\left(\frac{1}{7}\right) & =\left(L\left(\frac{4}{5},0\right)\oplus L\left(\frac{4}{5},3\right)\right)\otimes\left((L\left(\frac{6}{7},\frac{1}{7}\right)\oplus L\left(\frac{6}{7},\frac{22}{7}\right)\right)\\
			& \oplus L(\left(\frac{4}{5},\frac{2}{3}\right)^{+}\otimes L\left(\frac{6}{7},\frac{10}{21}\right)^{+}\oplus L\left(\frac{4}{5},\frac{2}{3}\right)^{-}\otimes L\left(\frac{6}{7},\frac{10}{21}\right)^{-},
		\end{align*}
		
		\begin{align*}
			\mathcal{V}\left(\frac{5}{7}\right) & =\left(L\left(\frac{4}{5},0\right)\oplus L\left(\frac{4}{5},3\right)\right)\otimes\left(L\left(\frac{6}{7},\frac{5}{7}\right)\oplus L\left(\frac{6}{7},\frac{12}{7}\right)\right)\\
			& \oplus L(\left(\frac{4}{5},\frac{2}{3}\right)^{+}\otimes L\left(\frac{6}{7},\frac{1}{21}\right)^{+}\oplus L\left(\frac{4}{5},\frac{2}{3}\right)^{-}\otimes L\left(\frac{6}{7},\frac{1}{21}\right)^{-},
		\end{align*}
		
		\begin{align*}
			\mathcal{V}\left(\frac{2}{5}\right) & =\left(L\left(\frac{4}{5},\frac{2}{5}\right)\oplus L\left(\frac{4}{5},\frac{7}{5}\right)\right)\otimes\left(L\left(\frac{6}{7},0\right)\oplus L\left(\frac{6}{7},5\right)\right)\\
			& \oplus L\left(\frac{4}{5},\frac{1}{15}\right)^{+}\otimes L(\left(\frac{6}{7},\frac{4}{3}\right)^{+}\oplus L\left(\frac{4}{5},\frac{1}{15}\right)^{-}\otimes L\left(\frac{6}{7},\frac{4}{3}\right)^{-},
		\end{align*}
		
		\begin{align*}
			\mathcal{V}\left(\frac{19}{35}\right) & =\left(L\left(\frac{4}{5},\frac{2}{5}\right)\oplus L\left(\frac{4}{5},\frac{7}{5}\right)\right)\otimes\left(L\left(\frac{6}{7},\frac{1}{7}\right)\oplus L\left(\frac{6}{7},\frac{22}{7}\right)\right)\\
			& \oplus L(\left(\frac{4}{5},\frac{1}{15}\right)^{+}\otimes L(\left(\frac{6}{7},\frac{10}{21}\right)^{+}\oplus L(\left(\frac{4}{5},\frac{1}{15}\right)^{-}\otimes L(\left(\frac{6}{7},\frac{10}{21}\right)^{-},
		\end{align*}
		\begin{align*}
			\mathcal{V}\left(\frac{39}{35}\right) & =\left(L\left(\frac{4}{5},\frac{2}{5}\right)\oplus L\left(\frac{4}{5},\frac{7}{5}\right)\right)\otimes\left(L\left(\frac{6}{7},\frac{5}{7}\right)\oplus L\left(\frac{6}{7},\frac{12}{7}\right)\right)\\
			& \oplus L(\left(\frac{4}{5},\frac{1}{15}\right)^{+}\otimes L(\left(\frac{6}{7},\frac{1}{21}\right)^{+}\oplus L(\left(\frac{4}{5},\frac{1}{15}\right)^{-}\otimes L\left(\frac{6}{7},\frac{1}{21}\right)^{-}.
		\end{align*}
		
	\end{lemma}

	It was proved in \cite{LYY} that $\mathcal{V}\subset\mathcal{U}$
	and as a module of $\mathcal{V}\otimes L\left(\frac{25}{28},0\right)$,
	
	\[
	\mathcal{U}_{6A}\cong\mathcal{V}\otimes L\left(\frac{25}{28},0\right)\oplus\mathcal{V}\left(\frac{1}{7}\right)\otimes L\left(\frac{25}{28},\frac{34}{7}\right)\oplus\mathcal{V}\left(\frac{5}{7}\right)\otimes L\left(\frac{25}{28},\frac{9}{7}\right).
	\]
	From here forward, we denote
	\begin{gather}
		P_{1}=\mathcal{V},\ P_{2}=\mathcal{V}\left(\frac{1}{7}\right),\ P_{3}=\mathcal{V}\left(\frac{5}{7}\right),\nonumber \\
		Q_{1}=L\left(\frac{25}{28},0\right),\ Q_{2}=L\left(\frac{25}{28},\frac{34}{7}\right),\ Q_{3}=L\left(\frac{25}{28},\frac{9}{7}\right),\label{PQ notation}
	\end{gather}
	and $U^{i}=P_{i}\otimes Q_{i}$, $i=1,2,3$. Then
	\[
	\mathcal{U}_{6A}\cong P_{1}\otimes Q_{1}\oplus P_{2}\otimes Q_{2}\oplus P_{3}\otimes Q_{3}=U^{1}\oplus U^{2}\oplus U^{3}.
	\]
	
	\begin{remark} \label{rationality-1} Since $\mathcal{V}\otimes L\left(\frac{25}{28},0\right)$
		is a rational and $C_{2}$-cofinite vertex operator algebra, it is
		straightforward to see that $\mathcal{U}$ is also rational and $C_{2}$-cofinite
		by \rm{\cite{ABD,HKL}}. \end{remark}
	
	\begin{remark} \label{self-dual}
		Since $\mathcal{U}_{1}=0$ and $\dim\mathcal{U}_{0}=1$ by Theorem
		\ref{bilinear form}, there is a unique bilinear form on $\mathcal{U}$
		and thus $\mathcal{U}'\cong\mathcal{U}$. Without loss of generality,
		we can identify $\mathcal{U}$ with $\mathcal{U}'$.
	\end{remark}   
	
	Now we recall the following results in \cite{DJY}.
	
	
	\begin{theorem}\label{uniqueness of 6A-algebra}
		The vertex operator algebra structure on $6A$-algebra $\mathcal{U}_{6A}$
		over $\mathbb{C}$ is unique.     
	\end{theorem}
	
	For simplicity, we shall use $[h_{1},\ h_{2}]$
	to denote the module $\mathcal{V}(h_{1})\otimes L(\frac{25}{28},\ h_{2})$ in this section.
	
	\begin{lemma} \label{realization of VOA}We have the following decomposition:
		\begin{align*}
			& \mathcal{L}\left(3,0\right)\otimes\mathcal{L}\left(1,0\right)\otimes\mathcal{L}\left(1,0\right)\otimes\mathcal{L}\left(1,0\right)\oplus\mathcal{L}\left(3,3\right)\otimes\mathcal{L}\left(1,1\right)\otimes\mathcal{L}\left(1,0\right)\otimes\mathcal{L}\left(1,0\right)\\
			& \cong\left\{ \left[0,\ 0\right]\oplus\left[\frac{1}{7},\frac{34}{7}\right]\oplus\left[\frac{5}{7},\frac{9}{7}\right]\right\} \otimes\mathcal{L}\left(6,0\right)\\
			& \oplus\left\{ \left[0,\frac{3}{4}\right]\oplus\left[\frac{5}{7},\frac{1}{28}\right]\oplus\left[\frac{1}{7},\frac{45}{28}\right]\right\} \otimes\mathcal{L}\left(6,2\right)\\
			& \oplus\left\{ \left[0,\frac{13}{4}\right]\oplus\left[\frac{1}{7},\frac{3}{28}\right]\oplus\left[\frac{5}{7},\frac{15}{28}\right]\right\} \otimes\mathcal{L}\left(6,4\right)\\
			& \oplus\left\{ \left[0,\frac{15}{2}\right]\oplus\left[\frac{1}{7},\frac{5}{14}\right]\oplus\left[\frac{5}{7},\frac{39}{14}\right]\right\} \otimes\mathcal{L}\left(6,6\right)
		\end{align*}
		Thus $\text{\ensuremath{\mathcal{L}}}\left(3,0\right)\otimes\mathcal{L}\left(1,0\right)\otimes\mathcal{L}\left(1,0\right)\otimes\mathcal{L}\left(1,0\right)\oplus\mathcal{L}\left(3,3\right)\otimes\mathcal{L}\left(1,1\right)\otimes\mathcal{L}\left(1,0\right)\otimes\mathcal{L}\left(1,0\right)$
		and $V_{L}$ contain a vertex operator subalgebra isomorphic to
		\[
		\left[0,\ 0\right]\oplus\left[\frac{1}{7},\frac{34}{7}\right]\oplus\left[\frac{5}{7},\frac{9}{7}\right]
		\]
		which is isomorphic to $\mathcal{U}_{6A}$ from the uniqueness of $\mathcal{U}_{6A}$.
	\end{lemma} 
	
	\begin{lemma} \label{Real modules} The following list give 14 irreducible
		$\mathcal{U}_{6A}$-module.
		\begin{align*}
			M^{0} & =\left[0,0\right]\oplus\left[\frac{1}{7},\frac{34}{7}\right]\oplus\left[\frac{5}{7},\frac{9}{7}\right], & M^{1} & =\left[0,\frac{3}{4}\right]\oplus\left[\frac{1}{7},\frac{45}{28}\right]\oplus\left[\frac{5}{7},\frac{1}{28}\right],\\
			M^{2} & =\left[0,\frac{13}{4}\right]\oplus\left[\frac{1}{7},\frac{3}{28}\right]\oplus\left[\frac{5}{7},\frac{15}{28}\right], & M^{3} & =\left[0,\frac{15}{2}\right]\oplus\left[\frac{1}{7},\frac{5}{14}\right]\oplus\left[\frac{5}{7},\frac{39}{14}\right],\\
			M^{4} & =\left[0,\frac{165}{32}\right]\oplus\left[\frac{1}{7},\frac{3}{224}\right]\oplus\left[\frac{5}{7},\frac{323}{224}\right], & M^{5} & =\left[0,\frac{5}{32}\right]\oplus\left[\frac{1}{7},\frac{675}{224}\right]\oplus\left[\frac{5}{7},\frac{99}{224}\right],\\
			M^{6} & =\left[0,\frac{57}{32}\right]\oplus\left[\frac{1}{7},\frac{143}{224}\right]\oplus\left[\frac{5}{7},\frac{15}{224}\right], & M^{7} & =\left[\frac{2}{5},0\right]\oplus\left[\frac{19}{35},\frac{34}{7}\right]\oplus\left[\frac{39}{35},\frac{9}{7}\right],\\
			M^{8} & =\left[\frac{2}{5},\frac{3}{4}\right]\oplus\left[\frac{19}{35},\frac{45}{28}\right]\oplus\left[\frac{39}{35},\frac{1}{28}\right], & M^{9} & =\left[\frac{2}{5},\frac{13}{4}\right]\oplus\left[\frac{19}{35},\frac{3}{28}\right]\oplus\left[\frac{39}{35},\frac{15}{28}\right],\\
			M^{10} & =\left[\frac{2}{5},\frac{15}{2}\right]\oplus\left[\frac{19}{35},\frac{5}{14}\right]\oplus\left[\frac{39}{35},\frac{39}{14}\right], & M^{11} & =\left[\frac{2}{5},\frac{5}{32}\right]\oplus\left[\frac{19}{35},\frac{675}{224}\right]\oplus\left[\frac{39}{35},\frac{99}{224}\right],\\
			M^{12} & =\left[\frac{2}{5},\frac{57}{32}\right]\oplus\left[\frac{19}{35},\frac{143}{224}\right]\oplus\left[\frac{39}{35},\frac{15}{224}\right], & M^{13} & =\left[\frac{2}{5},\frac{165}{32}\right]\oplus\left[\frac{19}{35},\frac{3}{224}\right]\oplus\left[\frac{39}{35},\frac{323}{224}\right].
		\end{align*}
	\end{lemma}

\subsection{Mirror extensions for rational vertex operator algebras}
A vertex operator algebra $V$ is said to be {\em regular} if any weak $V$-module
$M$ is a direct sum of irreducible $V$-modules.

We now assume that $V$ is a  vertex operator algebra satisfying the following conditions:

(1) $V$ is simple CFT type vertex operator algebra  and is self dual;

(2) $V$ is regular.

The following results were obtained in \cite{ABD}, \cite{L}.
\begin{theorem}
Let $V$ be a $CFT$ type vertex operator algebra. Then $V$ is regular if and only if $V$ is rational and $C_2$-cofinite.
\end{theorem}

We have the
following results which were proved in \cite{FHL}, \cite{DMZ} and \cite{ABD}.
\begin{theorem}\label{kvoa3}
Let $V^1, \cdots, V^p$ be  vertex operator algebras. If $V^1, \cdots, V^p$ are rational, then
$V^1\otimes \cdots \otimes V^p$  is rational and any irreducible $V^1\otimes \cdots \otimes V^p$-module is a tensor
product $M^1\otimes \cdots \otimes M^p$ , where $M^1, \cdots, M^p$ are
some irreducible modules for the vertex operator algebras $V^1,
\cdots, V^p,$ respectively. If $V^1, \cdots, V^p$ are $C_2$-cofinite, then $V^1\otimes \cdots \otimes V^p$  is $C_2$-cofinite.
In particular, if $V^1, \cdots, V^p$ are  vertex operator algebras satisfying conditions $(1)$ and $(2)$, then $V^1\otimes \cdots \otimes V^p$ satisfies conditions $(1)$ and $(2)$.
\end{theorem}    


We use the following theorem in \cite{Lin} to prove planty of essential results.

Let $V^1$, $V^2$ be  vertex operator algebras satisfying conditions (1) and (2). Denote  the module categories of $V^1$, $V^2$ by  $\mathcal{C}_{V^1}$, $\mathcal{C}_{V^2}$, respectively. 
If $V$ is an extension vertex operator algebra of $V^1\otimes V^2$, it is known that $V^1\otimes V^2$ is rational by \cite{FHL}, 
\cite{DMZ} and \cite{ABD},  then  $V$ has the following decomposition as $V^1\otimes V^2$-module:
$$V=\oplus_{i\in I, j\in J}Z_{i,j}M^i\otimes N^j,$$
where $Z_{i, j}$, $(i\in I, j\in J)$,  are nonnegative integers  and $\{M^i|i\in I\}$ (resp. $\{N^j|j\in J\}$) are inequivalent 
irreducible $V^1$-modules (resp. $V^2$-modules) such that $Z_{i, j}\neq 0$  for some $(i,j)\in I\times J$. 
In the following we  assume that $V$ is an  extension vertex operator algebra  of $V^1\otimes V^2$ such that $V$ is simple, self dual
and $\hom_{V^1\otimes V^2}(V^1\otimes N^j, V)=\mathbb{C}$ (resp. $\hom_{V^1\otimes V^2}(M^i\otimes V^2, V)=\mathbb{C}$) if and only if $N^j=V^2$ (resp. $M^i=V^1$).
Then we have

\begin{theorem}\label{c1}
(i) $Z_{i,j}=1$ if $Z_{i,j}\neq 0$. Moreover, this induces a bijection $\tau$ from $I$ to $J$ such that $\tau(i)=j$ if and only if $Z_{i, j}=1$. \\
(ii) The set $\{M^i| i\in I\}$ $($resp. $\{N^j| j\in J\})$ is closed under the tensor product of category $\mathcal{C}_{V^1}$ $($resp. $\mathcal{C}_{V^2})$. \\
(iii) For any $i_1, i_2,i_3\in I$,  $N_{M^{i_1}, M^{i_2}}^{M^{i_3}}=N_{(N^{\tau(i_1)})', (N^{\tau(i_2)})'}^{(N^{\tau(i_3)})'}$.
\end{theorem}

\section{Construction of $U_k^i$ and $\{\mathcal{U}_k\}$}
Let $\mathcal{U}_{0}=\mathcal{V}$ be the $3A$-algebra and $\mathcal{U}_1=\mathcal{U}_{\operatorname{6A}}$.
By GKO-construction and Lemma (\ref{realization of VOA}) we can write
$$\left(\mathcal{L}\left(3,0\right)\otimes\mathcal{L}\left(1,0\right)\otimes\mathcal{L}\left(1,0\right)\oplus\mathcal{L}\left(3,3\right)\otimes\mathcal{L}\left(1,1\right)\otimes\mathcal{L}\left(1,0\right)\right)\otimes\mathcal{L}(1,0)\cong\bigoplus_{i=0,\rm{even}}^6\mathcal{U}_{1,i}\otimes \mathcal{L}(6,i),$$
where $\mathcal{U}_{1,i}$ is defined by the summand before $\mathcal{L}(6,i)$ 
and, especially, $\mathcal{U}_{1,0}=\mathcal{U}_{1}$ is the commutant of $\mathcal{L}(6,0)$.

\begin{example}
It is shown in Lemma (\ref{realization of VOA}) that 
$$\mathcal{U}_1=\mathcal{U}_{1,0}= U_0^1\oplus U_0^3\oplus U_0^5
=\mathcal{V}\otimes L(\frac{25}{28},0)\oplus\mathcal{V}(\frac{5}{7})\otimes L(\frac{25}{28},\frac{9}{7})\oplus \mathcal{V}(\frac{1}{7})\otimes L(\frac{25}{28},\frac{34}{7}),$$
and
$$
\begin{aligned}
\mathcal{U}_{1,2}&=\left[0,\frac{3}{4}\right]\oplus\left[\frac{5}{7},\frac{1}{28}\right]\oplus\left[\frac{1}{7},\frac{45}{28}\right],\\
\mathcal{U}_{1,4}&=\left[0,\frac{13}{4}\right]\oplus\left[\frac{1}{7},\frac{3}{28}\right]\oplus\left[\frac{5}{7},\frac{15}{28}\right],\\
\mathcal{U}_{1,6}&=\left[0,\frac{15}{2}\right]\oplus\left[\frac{1}{7},\frac{5}{14}\right]\oplus\left[\frac{5}{7},\frac{39}{14}\right].
\end{aligned}
$$
\end{example}

For $k\geq 2$, $\mathcal{U}_{k}=\mathcal{U}_{k,0}$ is defined by
the commutant VOA of $\mathcal{L}(k+5,0)$ in GKO-construction,
and its modules $\mathcal{U}_{k,i}$ are defined by the summand paired with $\mathcal{L}(k+5,i)$.
That is, we write GKO-constructions as
$$\begin{aligned}
&\left(\mathcal{L}\left(3,0\right)\otimes\mathcal{L}\left(1,0\right)\otimes\mathcal{L}\left(1,0\right)\oplus\mathcal{L}\left(3,3\right)\otimes\mathcal{L}\left(1,1\right)\otimes\mathcal{L}\left(1,0\right)\right)\otimes\mathcal{L}(1,0)^{\otimes k}\\
\cong&\bigoplus_{i}\mathcal{U}_{k,i}\otimes \mathcal{L}(k+5,i).
\end{aligned}
$$
For simplicity, we define 
$\left[i-1,h\right]_k:=
\mathcal{U}_{k,i-1}\otimes L\left(c_{k+7},h\right)$ for odd $i$.

Specifically, we have
$$\mathcal{U}_{k,0}=\bigoplus_{i} U_{k-1}^{i}$$
where
$$U_k^{i}=\left[i-1,h_{(1,i)}^{(k+7)}\right]_k=
\mathcal{U}_{k,i-1}\otimes L\left(c_{k+7},h_{(1,i)}^{(k+7)}\right).$$


\begin{example}\label{U2}
We give a precise construction of $\mathcal{U}_2$.
Taking the commutant of $\mathcal{L}(7,0)$,
we have
$L(\frac{11}{12},0)$, $L(\frac{11}{12},\frac{5}{4})$, $L(\frac{11}{12},\frac{19}{4})$, and $L(\frac{11}{12},\frac{21}{2})$ 
with $(m,n)=(6,0),(6,2),(6,4)$, and $(6,6)$.
Similarly as Lemma (\ref{realization of VOA}), we have
\begin{equation*}
\begin{aligned}
&\mathcal{L}\left(3,0\right)\otimes\mathcal{L}\left(1,0\right)\otimes\mathcal{L}\left(1,0\right)\otimes\mathcal{L}\left(1,0\right)\otimes\mathcal{L}\left(1,0\right)\\
&\oplus\mathcal{L}\left(3,3\right)\otimes\mathcal{L}\left(1,1\right)\otimes\mathcal{L}\left(1,0\right)\otimes\mathcal{L}\left(1,0\right)\otimes\mathcal{L}\left(1,0\right)\\
\cong&\left\{\left[0,0\right]_1\oplus\left[2,\frac{5}{4}\right]_1\oplus\left[4,\frac{19}{4}\right]_1\oplus\left[6,\frac{21}{2}\right]_1\right\}\otimes\mathcal{L}\left(7,0\right)\\
&\oplus\left\{\left[0,\frac{7}{9}\right]_1\oplus\left[2,\frac{1}{36}\right]_1\oplus\left[4,\frac{55}{36}\right]_1\oplus\left[6,\frac{95}{18}\right]_1\right\}\otimes\mathcal{L}\left(7,2\right)\\
&\oplus\left\{\left[0,\frac{10}{3}\right]_1\oplus\left[2,\frac{7}{12}\right]_1\oplus\left[4,\frac{1}{12}\right]_1\oplus\left[6,\frac{11}{6}\right]_1\right\}\otimes\mathcal{L}\left(7,4\right)\\
&\oplus\left\{\left[0,\frac{23}{3}\right]_1\oplus\left[2,\frac{35}{12}\right]_1\oplus\left[4,\frac{5}{12}\right]_1\oplus\left[6,\frac{1}{6}\right]_1\right\}\otimes\mathcal{L}\left(7,6\right).
\end{aligned}
\end{equation*}

So
$$U_1^{i}=\left[i-1,h_{(1,i)}^{(8)}\right]_1=\mathcal{U}_{1,i-1}\otimes L\left(\frac{11}{12},h_{(1,i)}^{(8)}\right)$$
and we have
$$
\begin{aligned}
\mathcal{U}_{2}=&\mathcal{U}_{2,0}=U_1^1\oplus U_1^3\oplus U_1^5\oplus U_1^7
\\= &\left[0,0\right]_1
\oplus\left[2,\frac{5}{4}\right]_1\oplus\left[4,\frac{19}{4}\right]_1\oplus\left[6,\frac{21}{2}\right]_1,\\
\mathcal{U}_{2,2}=&\left[0,\frac{7}{9}\right]_1\oplus\left[2,\frac{1}{36}\right]_1\oplus\left[4,\frac{55}{36}\right]_1\oplus\left[6,\frac{95}{18}\right]_1,\\
\mathcal{U}_{2,4}=&\left[0,\frac{10}{3}\right]_1\oplus\left[2,\frac{7}{12}\right]_1\oplus\left[4,\frac{1}{12}\right]_1\oplus\left[6,\frac{11}{6}\right]_1,\\
\mathcal{U}_{2,6}=&\left[0,\frac{23}{3}\right]_1\oplus\left[2,\frac{35}{12}\right]_1\oplus\left[4,\frac{5}{12}\right]_1\oplus\left[6,\frac{1}{6}\right]_1.
\end{aligned}
$$
\end{example}


It is known by Theorem (\ref{uniqueness of 6A-algebra}) that $\mathcal{U}_{1}$ has a unique VOA structure and it satisfies conditions (1) and (2).
For $k\geq 2$, assume $\mathcal{U}_{k-1}$ have a unique VOA structure
and it satisfies conditions (1) and (2).
We will prove that there is a unique VOA structure on $\mathcal{U}_k$ satisfying these conditions.

\begin{remark}\label{rmk}
We simplify $\mathcal{U}_k$, $\mathcal{U}_{k-1,i}$, $L\left(c_{k+7},h_{(1,i)}^{(k+7)}\right)$ and $U_{k-1}^i$ 
by $\mathcal{U}$, $P^i$, $Q^i$ and $U^i$ respectively from now on,
unless otherwise specified.
\end{remark}

If there are no possible ambigious,
we will write $u^a$, or simply $a$, for elements in a module $U^a$.
We use $N,B,\mathcal{Y}$ (and $\tilde{N},\tilde{B},\tilde{\mathcal{Y}}$)
for fusion rules, braiding matrix, and intertwining operator basis
of $P^i$ (and $Q^i$).

\begin{theorem}\label{same fusion rule}
$N_{b,c}^a=\tilde{N}_{b,c}^a$.
\end{theorem}
\begin{proof}
Recall the notations in Theorem (\ref{c1}). Put $M^i=P^{i/2}$, $N^{i'}=Q^{2i-1}$, 
$V_1=\bigoplus_i M^i$, $V_2=\bigoplus_i N^i$,
$V=\mathcal{U}=\bigoplus_i P^i\otimes Q^i$.
Since $Q^{i'}=Q^i$,
we have $\tau=\operatorname{Id}$ and hence $N_{b,c}^a=\tilde{N}_{b,c}^a$.
\end{proof}

From this theorem, $P^i$ and $Q^i$ share same fusion rules.
Fix basis $\mathcal{Y}_{b,c}^a$ and $\tilde{\mathcal{Y}}_{b,c}^a$, 
then $\mathcal{I}_{b,c}^a=\mathcal{Y}_{b,c}^a\otimes \tilde{\mathcal{Y}}_{b,c}^a$ is a basis of
$I_{U^1}\ito{U^b}{U^c}{U^a}$.
So there exist constants $\lambda_{b,c}^a$ such that
$$Y(b\otimes\tilde{b},z)=\sum_{c,a}\lambda_{b,c}^a\mathcal{I}_{b,c}^a(b\otimes\tilde{b},z)=\sum_{c,a}\lambda_{b,c}^a\mathcal{Y}_{b,c}^a(b,z)\otimes\tilde{\mathcal{Y}}_{b,c}^a(\tilde{b},z).$$
We prove $\lambda_{b,c}^a\neq 0$ for all $a,b,c$ compatible with fusion rules in the next section.

\section{$\lambda_{i,j}^k\neq 0$}
Let $\mathcal C_P$ and $\mathcal C_Q$ be the module categories of $P^1$ and $Q^1$. 
Under the hypotheses of Section~3 these are semisimple rigid tensor categories.
In their Deligne product, we write $\otimes$ for the external tensor product and $\boxtimes$ for fusion.
Thus $U^i=P^i\otimes Q^i$, and $P^1$ and $Q^1$ are the tensor units.
The labels $i$ are the odd integers between $1$ and $m+1$, where $m$ is in \eqref{eq:(2.3)}.

Appendix~\ref{app:elementwise} gives the proof using four-point functions,
including the concrete computation of basis $e_{i,j}$ of intertwining operators and the matrix identity \eqref{braid matrix and lambda} used in Section~5.

\begin{theorem}\label{thm:all-nonzero}
$\lambda_{i,j}^k\neq 0$ for every $N_{i,j}^k =1$.
\end{theorem}
\begin{proof}
For each $b$, choose an evaluation $e_Q:Q^b\boxtimes Q^b\to Q^1$ and its coevaluation $d_Q:Q^1\to Q^b\boxtimes Q^b$ for the self-dual module $Q^b$. We suppress the associativity and unit isomorphisms. The duality maps satisfy
\[
(e_Q\boxtimes\id_{Q^b})(\id_{Q^b}\boxtimes d_Q)=\id_{Q^b}.
\]

Write $A=\mathcal U$. The vertex operator of this extension gives an associative multiplication $\mu:A\boxtimes A\to A$ with unit $U^1\hookrightarrow A$. For $N_{i,j}^k=1$, projection to $U^k$ gives
\[
\operatorname{pr}_{U^k}\circ\bigl(\mu|_{U^i\boxtimes U^j}\bigr)=\lambda_{i,j}^k\,\mathcal I_{i,j}^k.
\]

First consider the image $I_b=\mu(A\boxtimes U^b)$. It contains $U^b$ by the unit axiom. Tensoring preserves images by semisimplicity, so associativity gives
\[
\mu(A\boxtimes I_b)=\mu\bigl(\mu(A\boxtimes A)\boxtimes U^b\bigr)\subseteq I_b.
\]
Thus $u_n v\in I_b$ for $u\in A$, $v\in I_b$ and $n\in\mathbb Z$. Since $A$ is simple, its nonzero ideal $I_b$ equals $A$. In particular its projection onto $U^1$ is nonzero. For $i\ne b$, rigidity and self-duality give
\[
\operatorname{Hom}_{\mathcal C_Q}(Q^i\boxtimes Q^b,Q^1)
\cong\operatorname{Hom}_{\mathcal C_Q}(Q^i,Q^b)=0 .
\]
Only $U^b\boxtimes U^b$ can contribute to this projection. Hence $\lambda_{b,b}^1\ne0$.

Fix $a,b,r$ with $N_{a,b}^r=1$ and choose a summand inclusion $\iota:Q^r\hookrightarrow Q^a\boxtimes Q^b$. The chosen evaluation spans $\operatorname{Hom}_{\mathcal C_Q}(Q^b\boxtimes Q^b,Q^1)$. Write $\lambda_{b,b}^1\mathcal I_{b,b}^1=e_P\otimes e_Q$, where $e_P:P^b\boxtimes P^b\to P^1$ is nonzero. The epimorphism $e_P$ splits. Choose $s:P^1\to P^b\boxtimes P^b$ with $e_Ps=\id_{P^1}$. Set
\[
\begin{gathered}
T=(\id_{P^a}\boxtimes s)\otimes(\iota\boxtimes\id_{Q^b}),
\\
T:P^a\otimes(Q^r\boxtimes Q^b)\longrightarrow U^a\boxtimes U^b\boxtimes U^b .
\end{gathered}
\]
In $\operatorname{Hom}\bigl(P^a\otimes(Q^r\boxtimes Q^b),U^a\bigr)$, associativity gives
\[
\operatorname{pr}_{U^a}\circ(\mu|_{U^a\boxtimes A})\circ\bigl(\id_{U^a}\boxtimes(\mu|_{U^b\boxtimes U^b})\bigr)\circ T=\operatorname{pr}_{U^a}\circ(\mu|_{A\boxtimes U^b})\circ\bigl((\mu|_{U^a\boxtimes U^b})\boxtimes\id_{U^b}\bigr)\circ T.
\]

The first composite is
\[
P^a\otimes(Q^r\boxtimes Q^b)
 \xrightarrow{T} U^a\boxtimes(U^b\boxtimes U^b)
 \xrightarrow{\id_{U^a}\boxtimes(\mu|_{U^b\boxtimes U^b})} U^a\boxtimes A
 \xrightarrow{\mu|_{U^a\boxtimes A}} A
 \xrightarrow{\operatorname{pr}_{U^a}} U^a .
\]
For an intermediate summand $U^j$, the first tensor factor contains $P^1\xrightarrow{s}P^b\boxtimes P^b\to P^j$, which is zero unless $j=1$. Under $U^a\boxtimes U^1\simeq U^a$, the right unit axiom gives $\lambda_{a,1}^a\mathcal I_{a,1}^a=\id_{U^a}$. The surviving term is
\[
\begin{aligned}
\lambda_{a,1}^a\lambda_{b,b}^1\,\mathcal I_{a,1}^a(\id_{U^a}\boxtimes\mathcal I_{b,b}^1)T &=(\id_{U^a}\boxtimes(e_P\otimes e_Q))T\\
    &=\bigl((\id_{P^a}\boxtimes e_P)(\id_{P^a}\boxtimes s)\bigr)\otimes f\\
    &=\bigl(\id_{P^a}\boxtimes(e_Ps)\bigr)\otimes f\\
    &=(\id_{P^a}\boxtimes\id_{P^1})\otimes f\\
    &=\id_{P^a}\otimes f ,
\end{aligned}
\]
where
\[
f=(\id_{Q^a}\boxtimes e_Q)(\iota\boxtimes\id_{Q^b}):Q^r\boxtimes Q^b\longrightarrow Q^a .
\]
The triangle identity gives
\[
(f\boxtimes\id_{Q^b})(\id_{Q^r}\boxtimes d_Q)=\iota\ne0 .
\]

The second composite is
\[
P^a\otimes(Q^r\boxtimes Q^b)
 \xrightarrow{T} (U^a\boxtimes U^b)\boxtimes U^b
 \xrightarrow{(\mu|_{U^a\boxtimes U^b})\boxtimes\id_{U^b}} A\boxtimes U^b
 \xrightarrow{\mu|_{A\boxtimes U^b}} A
 \xrightarrow{\operatorname{pr}_{U^a}} U^a .
\]
The second tensor factor contains $Q^r\xrightarrow{\iota}Q^a\boxtimes Q^b\to Q^j$, which is zero unless $j=r$. The same index $j$ occurs in both factors of each summand of $A$. Restrict the associativity identity $\mu(\mu\boxtimes\id_A)=\mu(\id_A\boxtimes\mu)$ to $U^a\boxtimes U^b\boxtimes U^b$ and project to $U^a$. Expanding the intermediate copy of $A=\bigoplus_j U^j$ gives
\[
\begin{aligned}
&\sum_j\lambda_{j,b}^a\lambda_{a,b}^j\,\mathcal I_{j,b}^a(\mathcal I_{a,b}^j\boxtimes\id_{U^b})
&=\sum_j\lambda_{a,j}^a\lambda_{b,b}^j\,\mathcal I_{a,j}^a(\id_{U^a}\boxtimes\mathcal I_{b,b}^j),
\end{aligned}
\]
where each sum runs over the labels compatible with its two fusion rules. After precomposition with $T$, only $j=r$ survives on the left and only $j=1$ on the right. Thus
\[
\lambda_{r,b}^a\lambda_{a,b}^r\,\mathcal I_{r,b}^a(\mathcal I_{a,b}^r\boxtimes\id_{U^b})T=\lambda_{a,1}^a\lambda_{b,b}^1\,\mathcal I_{a,1}^a(\id_{U^a}\boxtimes\mathcal I_{b,b}^1)T=\id_{P^a}\otimes f\ne0 .
\]
It follows that $\lambda_{r,b}^a\lambda_{a,b}^r\ne0$ and hence $\lambda_{a,b}^r\ne0$.
\end{proof}

The theorem also gives the number of nonzero terms in the two four-point expansions. For inputs $U^a,U^b,U^d$ and output $U^r$, each admissible intermediate label contributes a product of two nonzero coefficients. The $Q^i$ are closed under fusion by Theorem~\ref{c1}(ii). Since the fusion multiplicities are at most one, the two counts are the dimensions of
\[
\begin{gathered}
\operatorname{Hom}_{\mathcal C_Q}\bigl(Q^a\boxtimes(Q^b\boxtimes Q^d),Q^r\bigr),
\qquad
\operatorname{Hom}_{\mathcal C_Q}\bigl(Q^b\boxtimes(Q^a\boxtimes Q^d),Q^r\bigr).
\end{gathered}
\]
Associativity and braiding identify these spaces. Similarly, self-duality and braiding preserve fusion compatibility under permutations of $(a,b,r)$, so all three coefficients $\lambda_{a,b}^r$, $\lambda_{b,a}^r$ and $\lambda_{b,r}^a$ are nonzero.

\section{The uniqueness of $(\mathcal{U},Y)$}
In this section, we fix basis $\{e_{i,j}\}$ and 
$Y(\cdot,z)$ a VOA structure on $\mathcal{U}$ such that $\lambda_{i,j}^k=1$
for all $i,j,k$ compatible with fusion rules.
Assume there is another VOA structure $\overline{Y}(\cdot,z)$ on $\mathcal{U}$,
and we will prove that the square of $\lambda_{b,c}^a$ for $\overline{Y}(\cdot,z)$
equals $1$ whenever $N_{b,c}^a=1$
if braiding matrix satisfies some non-vanishing conditions.
These conditions can be checked case by case for any fixed $\mathcal{U}_k$,
and we get $(\mathcal{U},\overline{Y}(\cdot,z))\cong (\mathcal{U},Y(\cdot,z))$,
thus proving the uniqueness of $\mathcal{U}$.

\begin{lemma}
After a diagonal change of coordinates fixing $U^1$ pointwise,
$\lambda_{1,k}^k=\lambda_{k,1}^k=\lambda_{k,k}^1=1$ for all $k$.
\end{lemma}
\begin{proof}    
Write $\mathcal U=\mathcal U_n$.
Since we assume $\mathcal{U}_{n-1}$ has a unique VOA structure,
we identify the prescribed $U^1$-module actions on the summands. Thus
we have $Y(u,z)=\overline{Y}(u,z)$ for all $u\in U^1$, that is $\lambda_{1,k}^k=1$.    

For any $u^{1}\in U^{1}$, $u^{k}\in U^{k}$, skew symmetry of $Y\left(\cdot,z\right)$
and $\overline{Y}\left(\cdot,z\right)$ ( \cite{FHL} ) imply

$$
\overline{Y}(u^{k},z)u^{1}=e^{zL\left(-1\right)}\overline{Y}\left(u^{1},-z\right)u^{k}=e^{zL\left(-1\right)}Y(u^{1},-z)u^{k}=Y\left(u^{k},z\right)u^{1}=\mathcal{I}_{k,1}^{k}\left(u^{k},z\right)u^{1}.
$$
In the mean time, $\overline{Y}\left(u^{k},z\right)u^{1}=\lambda_{k,1}^{k}\cdot\mathcal{I}_{k,1}^{k}(u^{k},z)u^{1}$.
Thus we get $\lambda_{k,1}^{k}=1$.

Note that by Remark \ref{self-dual}, $\mathcal{U}$ has a unique
invariant bilinear form $\left\langle \cdot,\cdot\right\rangle $
with $\left\langle 1,1\right\rangle =1$. For $u^{1}\in U^{1}$ and
$u^{k},v^{k}\in U^{k}$, we have
$$
\left\langle Y(u^{k},z)v^{k},u^{1}\right\rangle =\left\langle v^{k},Y\left(e^{zL\left(1\right)}\left(-z^{-2}\right)^{L\left(0\right)}u^{k},z^{-1}\right)u^{1}\right\rangle .
$$
That is,
$$
\left\langle \mathcal{I}_{k,k}^{1}\left(u^{k},z\right)v^{k},u^{1}\right\rangle 
=\left\langle v^{k},\mathcal{I}_{k,1}^{k}\left(e^{zL\left(1\right)}\left(-z^{-2}\right)^{L\left(0\right)}u^{k},z^{-1}\right)u^{1}\right\rangle .
$$
Let $\langle\cdot,\cdot\rangle_{\overline Y}$ be the vacuum-normalized form for $\overline Y$.
The inequivalent self-dual summands are orthogonal for both forms. Their restrictions to $U^k$ are nondegenerate, so Schur's lemma gives
\[
 \langle u,v\rangle_{\overline Y}=a_k\langle u,v\rangle
 \quad(u,v\in U^k),\qquad a_k\ne0,\quad a_1=1.
\]
Choose $d_1=1$ and $d_k^2a_k=1$, and set $T|_{U^k}=d_k\,\id$.
Then $\overline Y^T(u,z)v=T^{-1}\overline Y(Tu,z)Tv$ has invariant form $\langle Tu,Tv\rangle_{\overline Y}=\langle u,v\rangle$.
We continue to write $\overline Y$ for this representative.

The invariant bilinear form on $\left(\mathcal{U},\overline{Y}\right)$
gives
$$
\left\langle \lambda_{k,k}^{1}\cdot\mathcal{I}_{k,k}^{1}\left(u^{k},z\right)v^{k},u^{1}\right\rangle 
=\left\langle v^{k},\lambda_{k,1}^{k}\cdot\mathcal{I}_{k,1}^{k}\left(e^{zL\left(1\right)}\left(-z^{-2}\right)^{L\left(0\right)}u^{k},z^{-1}\right)u^{1}\right\rangle .
$$
Since $\lambda_{k,1}^k=1$, we get $\lambda_{k,k}^{1}=1$.
\end{proof}

\begin{theorem}
$\lambda_{i,j}^k=\lambda_{j,i}^k=\lambda_{j,k}^i$ for all $i,j,k$ compatible with fusion rules.
\end{theorem}
\begin{proof}
This comes from the proof of Theorem (\ref{alternation}) and previous lemma.
The one-dimensional braiding factors are the same for $Y$ and $\overline Y$ and cancel on comparing the two identities. Skew symmetry interchanges the first two indices.
\end{proof}

\begin{lemma}\label{partunitary}
$B^{\operatorname{T}}\tilde{B}=I$.
\end{lemma}
\begin{proof}
Recall (\ref{braid matrix and lambda}).
This is the direct consequence from the $Y(\cdot,z)$ we chose.
\end{proof}



Fusion rules show that the longest direct summand appears at the $(2K+1)$st position when $m=4K+1$ and at the $(2K+3)$rd position when $m=4K+3$.
When $m=4K$, the longest direct summand appears at the $(2K+1)$st position. 
When $m=4K+2$, the longest direct summand appear at both the $(2K+1)$st and $(2K+3)$rd position,
and in this paper we use the $(2K+1)$st one.
We denote the module with the longest fusion rule by $U^t$.


\begin{lemma}
$\left(\lambda_{t,t}^a\right)^2=1$ if $\left(\tilde{B}_{a,t}^{a,t}\right)_{U^1,U^t}\neq 0$.
\end{lemma}
\begin{proof}
Consider $(a,b,t,t)$ in (\ref{braid matrix and lambda}), and we have
$$
\operatorname{diag}(\lambda_{a,b}^{\beta_1}\lambda_{t,t}^{\beta_1},\ldots,\lambda_{a,b}^{\beta_N}\lambda_{t,t}^{\beta_N})\tilde{B}
=\tilde{B}\operatorname{diag}(\lambda_{a,t}^{\mu_1}\lambda_{b,t}^{\mu_1},\ldots,\lambda_{a,t}^{\mu_N}\lambda_{b,t}^{\mu_N}).$$
Compare the $(i,j)$-element in the matrix of each side of the euqation,
and we have 
\begin{equation}\label{abtt}
\lambda_{a,b}^{\beta_i}\lambda_{t,t}^{\beta_i}\tilde{B}_{i,j}=\tilde{B}_{i,j}\lambda_{a,t}^{\mu_j}\lambda_{b,t}^{\mu_j}
\end{equation}
for all $i$ and $j$.
Choose $(i,j)$ such that $\beta_i=U^1$ and $\mu_j=U^t$.
It is possible by computing the fusion rules.
Since $\lambda_{t,t}^1=1$ and $b=a$, 
we get $\tilde{B}_{i,j}=\tilde{B}_{i,j}\left(\lambda_{a,t}^{\mu_{j}}\right)^2$.
So we have $\left(\lambda_{t,t}^a\right)^2=1$ if 
$\left(\tilde{B}_{a,t}^{a,t}\right)_{U^1,U^t}\neq 0$. 
\end{proof}

\begin{theorem}\label{squarelambda}
For all $i,j,k$ compatible with fusion rules, 
$\left(\lambda_{p,q}^r\right)^2=1$ if 
$\left(\tilde{B}_{p,t}^{q,t}\right)_{U^r,U^t}\neq 0$.
\end{theorem}
\begin{proof}
Choose $j$ such that $\mu_j=U^t$ in (\ref{abtt}),
and go through all $\beta_i$ such that $N_{a,b}^{\beta_i}=1$.
We get what we want by multiplying $\lambda_{t,t}^{\beta_i}$ on both sides of the equation
and setting $(p,q,r)=(a,b,\beta_i)$.
\end{proof}

\begin{remark}\label{rmkB}
From now on if we use 
elements in $\tilde{B}$, we will assume they are non-zero.
For any fixed $\mathcal{U}_k$, it can be checked case by case. 
Examples of computing elements of braiding matrix are given at the end of this section.

\end{remark}



\begin{lemma}
Define a map
$\sigma:(\mathcal{U},Y)\to(\mathcal{U},\overline{Y})$ 
by
\begin{itemize}
\item If $N_{t,t}^k=1$, then $\sigma(u^k):=\lambda_{t,t}^k u^k$;
\item If $N_{t,t}^k=0$, then $\sigma(u^k):=\lambda_{a,b}^k \sigma(u^a)_{(n)}\sigma(u^b)\restriction U^k$ for any $a,b<k$ with $N_{a,b}^k=1$ and $u^a_{(n)}u^b\restriction U^k=u^k$ (Here $\restriction$ is the restriction),
\end{itemize} 
and linear expansion.
Then it is a well-defined homomorphism.
\end{lemma}
\begin{proof}
We write $u^p_{(n)}u^q=\sum_r w^r$,
and we write the $n$th product of $\overline{Y}$ as $(\overline{n})$.
By the definition of $\lambda_{p,q}^r$ we have $u^p_{(\overline{n})}u^q=\sum_r \lambda_{p,q}^r w^r$.
By Lemma~5.4 and the assumption in Remark~5.6, $(\lambda_{t,t}^a)^2=1$ whenever $(t,t,a)$ is compatible with fusion rules.
Note that from (\ref{abtt}) and by putting $\mu_j=t$,
we have, for $(a,b,\beta_i)$, $(t,t,\beta_i)$, $(t,t,a)$ and $(t,t,b)$ compatible with fusion rules, 
\begin{equation}\label{lambdat}
\lambda_{a,b}^{\beta_i}\lambda_{t,t}^{\beta_i}=\lambda_{a,t}^{t}\lambda_{b,t}^{t}.
\end{equation}
Indeed, these fusion rules give row $\beta_i$ and column $t$ in $(a,b,t,t)$, and $\tilde B_{i,j}\ne0$ by assumption. By Theorem~5.2 and Lemma~5.4 we get
\[
 \lambda_{a,b}^{\beta_i}=\lambda_{t,t}^a\lambda_{t,t}^b\lambda_{t,t}^{\beta_i}.
\]


If $m\neq 4K+2$, then it is well-defined since the fusion rule that $N_{t,t}^k=1$ for all odd $k$.
If $m=4K+2$, then $t=2K+1$ and $k=m+1$ is the only case when $N_{t,t}^k=0$.
Assume there are $a',b'<k$ such that $N_{a',b'}^k=1$ and $u^{a'}_{(n)}u^{b'}\restriction U^k=u^k$.
We need to prove $\lambda_{a,b}^k\lambda_{t,t}^a\lambda_{t,t}^b =\lambda_{a',b'}^k\lambda_{t,t}^{a'}\lambda_{t,t}^{b'},$
where $a,b<k$. Since $k=m+1$, from fusion rules we get $b=m+2-a$ and $b'=m+2-a'$.
Consider $(m+1,m+1,a,a)\sim(m+1,a,m+1,a)$, and we have
\[
 (\lambda_{a,m+2-a}^{m+1})^2
 =\lambda_{a,m+2-a}^{m+1}\lambda_{m+1,a}^{m+2-a}
 =\lambda_{m+1,1}^{m+1}\lambda_{a,a}^1=1.
\]
For odd $a$ with $3\leq a\leq m-3$, consider $(m+1,a,a+2,m-1)\sim(m+1,a+2,a,m-1)$, which gives
\[
 \lambda_{a,m+2-a}^{m+1}\lambda_{a+2,m-1}^{m+2-a}
 =\lambda_{a+2,m-a}^{m+1}\lambda_{a,m-1}^{m-a}.
\]
Applying (\ref{lambdat}) gives
\[
 \lambda_{a,m+2-a}^{m+1}\lambda_{t,t}^a\lambda_{t,t}^{m+2-a}
 =\lambda_{a+2,m-a}^{m+1}\lambda_{t,t}^{a+2}\lambda_{t,t}^{m-a}.
\]
By induction on $a=3,5,\ldots,m-1$, the product on the left always equals $\lambda_{3,m-1}^{m+1}\lambda_{t,t}^3\lambda_{t,t}^{m-1}$.
So we get the well-defined map.

Now we check $\sigma(u^p_{(n)}u^q)=\sigma(u^p)_{(\overline{n})}\sigma(u^q)$.
We have $\sigma(u^p_{(n)}u^q)=\sum_r \lambda_{t,t}^r w^r$, 
and $\sigma(u^p)_{(\overline{n})}\sigma(u^q)=\lambda_{t,t}^p\lambda_{t,t}^q\sum_r \lambda_{p,q}^r w^r$.
We get the equation by (\ref{lambdat}) and Theorem (\ref{squarelambda}).

\end{proof}

\begin{theorem}\label{unique}
The VOA structure of $\mathcal{U}$ is unique 
if corresponding elements in braiding matrices are non-zero.
\end{theorem}
\begin{proof}
The $\sigma$ from the previous lemma is the isomorphism we want.
It is the direct consequence of previous lemma and Theorem (\ref{squarelambda}).
\end{proof}

Note that we have (\ref{abtt}) and Remark (\ref{rmkB}), 
so it is enough to check elements of $\tilde{B}$ are non-zero.
We state examples of computing $\tilde{B}$, 
especially when $\mathcal{U}=\mathcal{U}_1=\mathcal{U}_{\operatorname{6A}}$,
and it can be found in \cite{DJY}. And we can get the uniqueness of $\mathcal{U}_2$ as an example
by direct computation following Proposition (\ref{fusion rules of virasoro modules}) and (\ref{FFK 2.19}).

\begin{example}
Consider braiding matrix for $L\left(25/28,0\right)$-modules.
Denote irreducible $L\left(25/28,0\right)$-modules $L\left(25/28,34/7\right)$
and $L\left(25/28,9/7\right)$ by $Q_{2}$ and $Q_{3}$
respectively. For convenience, we will denote $\left(\tilde{B}_{Q_{a},Q_{b}}^{Q_{c},Q_{d}}\right)_{Q_{e},Q_{f}}$
by $\left(\tilde{B}_{a,b}^{c,d}\right)_{e,f}$, $a,b,c$, $d,$ $e$,
$f$ $\in\left\{ 2,3\right\} $. 
Then we have
\label{(2332)_3,2 nonzero} \label{(3322)_3,3 nonzero}
\label{(2333)_3,2 nonzero} \label{(3,2,3,3)_2,3 nonzero}$\left(\tilde{B}_{2,2}^{3,3}\right)_{3,2}\not=0$,
$\left(\tilde{B}_{3,2}^{3,2}\right)_{3,3}\not=0$, $\left(\tilde{B}_{2,3}^{3,3}\right)_{3,2}\not=0$,
and $\left(\tilde{B}_{3,3}^{2,3}\right)_{2,3}\not=0$.
\end{example}

\begin{example}
We use notations in Remark (\ref{rmk}), and consider $L(11/12, 0)$-modules.
In this case $t=3$,
and we have
$\left(\tilde{B}_{a,3}^{b,3}\right)_{3,c}\neq 0$
for all $a,b,c\in \{3,5,7\}$ such that those $\tilde{B}_{i,j}$ 
appear in Theorem (\ref{squarelambda}) and compatible with fusion rules.

We write the braiding matrix as $B_{(a,b,c)}$.
We need to check all the 
$$(a,b,c)\in\left\{\begin{aligned}
&(3,3,3),(3,3,5),(3,5,3),(3,5,5),\\
&(5,3,3),(5,3,5),(5,5,3),(5,5,5),\\
&(5,5,7),(5,7,5),(7,5,5)
\end{aligned}\right\}.$$

To prove $B_{(3,3,3)}=\left(\tilde{B}_{3,3}^{3,3}\right)_{3,3}\neq 0$, 
we need to check 
$$
r(3,3,3,3)_{3,3}=r(3,2,3,2)_{2,3}r(2,2,3,3)_{3,2}+r(3,2,3,4)_{2,3}r(2,2,3,3)_{3,4}\neq 0.
$$
Since 
$$\begin{aligned}
r(3,2,3,2)_{2,3}&=r(3,2,2,1)_{2,2}r(2,2,2,2)_{1,3}=x^{1/4}x^{-1/4}\frac{[1]}{[2]},\\
r(2,2,3,3)_{3,2}&=r(2,2,2,2)_{3,1}r(1,2,2,3)_{2,2}+r(2,2,2,2)_{3,3}r(3,2,2,3)_{2,2},\\
&=x^{-1/4}\frac{[3]}{[2]}x^{1/4}+\left(-x^{-1/4-2/2}\frac{[1]}{[2]}x^{-1/4+3/2}\frac{[1]}{[3]}\right),\\
r(3,2,3,4)_{2,3}&=r(3,2,2,3)_{2,2}r(2,2,2,4)_{3,3}+r(3,2,2,3)_{2,4}r(4,2,2,4)_{3,3}\\
&=x^{-1/4+3/2}\frac{[1]}{[3]}x^{1/4}+x^{-1/4}\frac{[2]}{[3]}x^{-1/4+4/2}\frac{[1]}{[4]},\\
r(2,2,3,3)_{3,4}&=r(2,2,2,2)_{3,3}r(3,2,2,3)_{2,4}\\
&=-x^{-1/4-2/2}\frac{[1]}{[2]}x^{-1/4}\frac{[2]}{[3]},
\end{aligned}$$
let $r=x^{1/2}+x^{-1/2}$, and we get 
$$
\begin{aligned}
B_{(3,3,3)}
&=\frac{[1]}{[2]}\left(\frac{[3]}{[2]}-\frac{[1][1]}{[2][3]}\right)-\left(\frac{[1]}{[3]}+\frac{[1][2]}{[3][4]}\right)\frac{[1]}{[3]}\\
&=1-\frac{1}{r^2-1}-\frac{1}{(r^2-1)^2}-\frac{1}{(r^2-1)^2(r^2-2)}\\
&=1-\frac{1}{r^2-2}.
\end{aligned}
$$
That is, $\operatorname{Re}B_{(3,3,3)}> 0$, so $B_{(3,3,3)}\neq 0$

Following the same approach above, we have
$$\begin{aligned}
\operatorname{Im}B_{(3,3,5)}>0,&
\operatorname{Re}B_{(3,5,3)}<0,&
\operatorname{Im}B_{(3,5,5)}>0,&
\operatorname{Im}B_{(5,3,3)}>0,&
\operatorname{Re}B_{(5,3,5)}>0,\\
\operatorname{Im}B_{(5,5,3)}>0,&
\operatorname{Re}B_{(5,5,5)}<0,&
\operatorname{Re}B_{(5,5,7)}>0,&
\operatorname{Re}B_{(5,7,5)}<0,&
\operatorname{Re}B_{(7,5,5)}<0.
\end{aligned}$$


So all the elements in braiding matrix we need are non-zero. 
Thus $\mathcal{U}_2=U_1^1\oplus U_1^3\oplus U_1^5\oplus U_1^7=
\left[0,0\right]_1\oplus\left[2,\frac{5}{4}\right]_1
\oplus\left[4,\frac{19}{4}\right]_1\oplus\left[6,\frac{21}{2}\right]_1$
has a unique VOA structure
by Theorem (\ref{squarelambda}) and the previous theorem.
\end{example}

\section{Generated by Griess algebra}
It is proved in \cite{JZ} that $\mathcal{U}_1$ is generated by Griess algebra.
Now we prove that if $\mathcal{U}_{k-1}$ is generated by Griess algebra, so is $\mathcal{U}=\mathcal{U}_k$.
\begin{lemma}
$U^3_{k-1}$ is generated by Griess algebra.
\end{lemma}
\begin{proof}
Since $U^3_{k-1}$ is an irreducible $U^1$-module,
we only need to find one vector with weight $2$ in it.

For $\mathcal{U}_1$, 
we have $U^{3}_0=\mathcal{V}(\frac{5}{7})\otimes L(\frac{25}{28},\frac{9}{7})$,
and $\frac{5}{7}+\frac{9}{7}=2.$

For $\mathcal{U}_k$ with $k\geq 2$, we have
$U^3_{k-1}=\mathcal{U}_{k-1,2}\otimes L\left(c_{k+6},h_{(1,3)}^{k+6}\right)$
and 
$$\mathcal{U}_{k-2,0}
\otimes L\left(c_{k+5},h_{(3,1)}^{k+5}\right)
=\left[0,h_{(3,1)}^{k+5}\right]_{k-2}\subseteq \mathcal{U}_{k-1,2}.$$
Note that
$$
h^{(k+6)}_{(1,3)}+ h^{(k+5)}_{(3,1)}=\frac{k+8}{k+6}+\frac{k+4}{k+6}=2,
$$
and the proof is complete.
\end{proof}

\begin{theorem}
$\mathcal{U}$ is generated by Griess algebra.
\end{theorem}
\begin{proof}
We already know $U^1$ and $U^3$ are generated by Griess algebra.
Assume $U^{2k+1}$ are generated by Griess algebra for $0\leq k\leq n$.
By fusion rules of $U^3$ we have $U^{2k+3}\subseteq U^3\boxtimes U^{2k+1}$,
so $U^{2n+3}$ is generated by Griess algebra.
Thus $\mathcal{U}$ is generated by Griess algebra.
\end{proof}

\newpage

\appendix
\section{An elementwise proof of nonvanishing}\label{app:elementwise}
When two constant $\lambda$ and $\mu$ are either both non-zero or both zero, 
we will write the relation as $\lambda\sim \mu$.
The goal of this appendix is proving $\lambda_{i,j}^k\sim 1$ for all $i,j,k$ compatible with fusion rules, that is, $\lambda_{i,j}^k\neq 0$ if $N_{i,j}^k=1$.

We denote tuple $(i,j,k,l)$ for 
$$\left\langle {u^i}',Y(u^j,z_1)Y(u^k,z_2)u^l\right\rangle,$$
and we write Proposition (\ref{extension property}) as
$(i,j,k,l)\sim(i,k,j,l)$.

Note that the two $\sim$ have different meanings, 
and we can infer the correct one by context.
In the lists of coefficients below, an inadmissible intermediate label contributes zero and is omitted before nonzero terms are replaced by $1$.

\subsection{Fundamental Results}
We denote $e_{i,j}$ for $\mathcal{Y}_{c,i}^a(\cdot,z_2)\mathcal{Y}_{b,d}^{i}(\cdot,z_1)\otimes \tilde{\mathcal{Y}}_{c,j}^a(\cdot,z_2)\tilde{\mathcal{Y}}_{b,d}^{j}(\cdot,z_1)$.
Decompose $Y$ to the basis $\{e_{i,j}:i,j=1,\ldots,N\}$,
by direct computation we have
$$
\begin{aligned}
&E\left\langle {(a\otimes \tilde{a})}',
Y(b\otimes \tilde{b},z_1)
Y(c\otimes \tilde{c},z_2)
d\otimes \tilde{d}\right\rangle\\
=&E\left\langle {(a\otimes \tilde{a})}',
\sum_{\beta}\lambda_{b,\beta}^a\lambda_{c,d}^{\beta}
\mathcal{Y}_{b,\beta}^a
\otimes\tilde{\mathcal{Y}}_{b,\beta}^a(b\otimes \tilde{b},z_1)
\mathcal{Y}_{c,d}^{\beta}
\otimes\tilde{\mathcal{Y}}_{c,d}^{\beta}(c\otimes \tilde{c},z_2)
d\otimes \tilde{d}\right\rangle\\
=&E\left\langle {(a\otimes \tilde{a})}',
\sum_{\beta}\lambda_{b,\beta}^a\lambda_{c,d}^{\beta}
\mathcal{Y}_{b,\beta}^a (b,z_1)
\mathcal{Y}_{c,d}^{\beta}(c,z_2)d
\otimes
\tilde{\mathcal{Y}}_{b,\beta}^{a}(\tilde{b},z_1)
\tilde{\mathcal{Y}}_{c,d}^{\beta}(\tilde{c},z_2)
\tilde{d}\right\rangle
\\
=&E\left\langle {(a\otimes \tilde{a})}',
\sum_{\beta}\lambda_{b,\beta}^a\lambda_{c,d}^{\beta}
\sum_{\gamma_1}B_{\beta,\gamma_1}
\mathcal{Y}_{c,\gamma_1}^a(c,z_2)\mathcal{Y}_{b,d}^{\gamma_1}(b,z_1)d
\otimes\sum_{\gamma_2}\tilde{B}_{\beta,\gamma_2}
\tilde{\mathcal{Y}}_{c,\gamma_2}^a(\tilde{c},z_2)\tilde{\mathcal{Y}}_{b,d}^{\gamma_2}(b,z_1)\tilde{d}
\right\rangle
\\
=&E\left\langle {(a\otimes \tilde{a})}',
\sum_{\beta,\gamma_1,\gamma_2}\lambda_{b,\beta}^a\lambda_{c,d}^{\beta}
B_{\beta,\gamma_1}\tilde{B}_{\beta,\gamma_2}
e_{\gamma_1,\gamma_2}d\otimes \tilde{d}\right\rangle
\end{aligned}
$$
and
$$
\begin{aligned}
&E\left\langle {(a\otimes \tilde{a})}',
Y(c\otimes \tilde{c},z_2)Y(b\otimes \tilde{b},z_1)d\otimes \tilde{d}\right\rangle\\
=&E\left\langle {(a\otimes \tilde{a})}',
\sum_{\mu}\lambda_{c,\mu}^a\lambda_{b,d}^{\mu} 
\mathcal{Y}_{c,\mu}^a(c,z_2)\mathcal{Y}_{b,d}^{\mu}(b,z_1)d
\otimes 
\tilde{\mathcal{Y}}_{c,\mu}^a(\tilde{c},z_2)\tilde{\mathcal{Y}}_{b,d}^{\mu}(\tilde{b},z_1)\tilde{d}\right\rangle\\
=&E\left\langle {(a\otimes \tilde{a})}',
\sum_{\mu}\lambda_{c,\mu}^a\lambda_{b,d}^{\mu} 
e_{\mu,\mu}d\otimes\tilde{d}\right\rangle,
\end{aligned}
$$
all possible $\beta,\gamma_1,\gamma_2,\mu$ compatible with fusion rules.

\begin{remark}
Note that elements in braiding matrix is denoted by $\left(B_{a_{4,}a_{1}}^{a_{3},a_{2}}\right)_{\mu,\gamma}^{i,j;k,l}$ in (\ref{braiding matrix property}).
In this paper, $i,j,k,l=1$ by fusion rules, and $a_1,a_2,a_3,a_4$ can always be inferred from the context.
So we can simplify elements in braiding matrix as $B_{\mu,\gamma}$.

Rather than focusing on elements, 
we are going to write the entire braiding matrices.
We need to determine elements by its position in the matrix, 
and we also need to determine the elements by corresponding modules.

To clarify notations, 
we use $B_{i,j}$ to determine the $(i,j)$-element in the matrix,
and $B_{U^i,U^j}$ or $B_{\mu,\gamma}$ the element whose position in the braiding matrix
is determined by (\ref{braiding matrix property}).
\end{remark}


\begin{lemma}
The relation $(a,b,c,d)\sim(a,c,b,d)$
can be written as
\begin{equation}\label{braid matrix and lambda}
B^{\operatorname{T}}\operatorname{diag}\left(\lambda_{b,\beta_1}^a\lambda_{c,d}^{\beta_1},\ldots,\lambda_{b,\beta_N}^a\lambda_{c,d}^{\beta_N}\right)\tilde{B}
=\operatorname{diag}\left(\lambda_{c,\mu_1}^a\lambda_{b,d}^{\mu_1},\ldots,\lambda_{c,\mu_N}^a\lambda_{b,d}^{\mu_N}\right). 
\end{equation}
where $B=\left(B_{i,j}\right)$,
$\tilde{B}=\left(\tilde{B}_{i,j}\right)$,
and $\mu_i, \beta_j \left(i,j=1,\ldots,N\right)$ are all modules compatible with fusion rules.
\end{lemma}
\begin{proof}
Let $C=\operatorname{diag}\left(\lambda_{b,\beta_1}^a\lambda_{c,d}^{\beta_1},\ldots,\lambda_{b,\beta_N}^a\lambda_{c,d}^{\beta_N}\right)$,
$D=\operatorname{diag}\left(\lambda_{c,\mu_1}^a\lambda_{b,d}^{\mu_1},\ldots,\lambda_{c,\mu_N}^a\lambda_{b,d}^{\mu_N}\right)$,
and $A=B^{\operatorname{T}}C\tilde{B}$.
Now we check the $A_{i,j}=D_{i,j}$.
By direct computation we know
$$A_{i,j}=\sum_{n,m=1}^N B_{n,i}C_{n,m}\tilde{B}_{m,j}=\sum_{k=1}^N B_{k,i}C_{k,k}\tilde{B}_{k,j}.$$ 

Consider $(a,b,c,d)\sim(a,c,b,d)$.
Use the linear independence of four point functions,
we get 
$$
\sum_{\beta,\gamma_1,\gamma_2}\lambda_{b,\beta}^a\lambda_{c,d}^{\beta}
B_{\beta,\gamma_1}\tilde{B}_{\beta,\gamma_2}
e_{\gamma_1,\gamma_2}
=\sum_{\mu}\lambda_{c,\mu}^a\lambda_{b,d}^{\mu} 
e_{\mu,\mu}.
$$
LHS of it are non-zero only if $\gamma_1=\gamma_2=\mu$.
Since $e_{i,j}$ are basis,
we have 
\begin{equation}\label{DJY equations}
	\sum_{\beta}\lambda_{b,\beta}^a\lambda_{c,d}^{\beta}B_{\beta,\mu_i}\tilde{B}_{\beta,\mu_i}
=\lambda_{c,\mu_i}^a\lambda_{b,d}^{\mu_i}
\end{equation}
for all possible $\mu_i$ compatible with fusion rules.

So we get $A_{i,j}=0=D_{i,j}$ if $i\neq j$.
And we have
$$D_{i,i}=\lambda_{c,\mu_i}^a\lambda_{b,d}^{\mu_i}
=\sum_{\beta}\lambda_{b,\beta}^a\lambda_{c,d}^{\beta}B_{\beta,\mu_i}\tilde{B}_{\beta,\mu_i}
=\sum_{k=1}^N C_{k,k}B_{k,i}\tilde{B}_{k,i}=A_{i,i},$$
thus completing the proof.
\end{proof}


	

\begin{remark}
	Instead of (\ref{DJY equations}),
	\cite{DJY} use system of equations in Theorem 3.9.
\end{remark}

%
From the previous lemma
we get
\begin{equation}\label{braid matrix and lambda sum}
\left(\lambda_{c,\mu_1}^a\lambda_{b,d}^{\mu_1},\ldots,\lambda_{c,\mu_N}^a\lambda_{b,d}^{\mu_N}\right)
=\left(\sum_{j=1}^N\lambda_{b,\beta_j}^a\lambda_{c,d}^{\beta_j}B_{\beta_j,\mu_1}\tilde{B}_{\beta_j,\mu_1},\ldots,\sum_{j=1}^N\lambda_{b,\beta_j}^a\lambda_{c,d}^{\beta_j}B_{\beta_j,\mu_N}\tilde{B}_{\beta_j,\mu_N}\right)
\end{equation}


Since we don't need the order of $\mu_i$ and $\beta_j$,
and the values of $B_{i,j}$ and $\tilde{B}_{i,j}$ are irrelevant in this section,
we write the equation (\ref{braid matrix and lambda sum}) which comes fom $(a,b,c,d)\sim(a,c,b,d)$ 
as a relation between two sets of numbers:
\begin{equation}\label{lambda equation}
\left\{\lambda_{c,\mu_1}^a\lambda_{b,d}^{\mu_1},\ldots,\lambda_{c,\mu_N}^a\lambda_{b,d}^{\mu_N}\right\}
\sim\left\{\lambda_{b,\beta_1}^a\lambda_{c,d}^{\beta_1},\ldots,\lambda_{b,\beta_N}^a\lambda_{c,d}^{\beta_N}\right\}.
\end{equation}

\begin{theorem}\label{samerank}
The none-zero elements appear same times in each side of (\ref{lambda equation}).
\end{theorem}
\begin{proof}
Noting that $B,\tilde{B}$ are invertible and two sides of (\ref{braid matrix and lambda}) share the same rank, the proof is immediate.
\end{proof}

\begin{remark}
Note that if there is only one element in (\ref{lambda equation}), this relation will be compatible with the $\sim$ between numbers.
In this case we will omit the braces.
\end{remark}

In order to get the uniqueness of $\mathcal{U}_k$,
we need to know some elements in $B$ and $\tilde{B}$.
We show in Section~5 that $B$ can be computed from $\tilde{B}$,
and non-zero elements $\tilde{B}_{i,j}$ can be checked cace by case for any fixed $\mathcal{U}_k$.

Now we are ready to start proving $\lambda_{i,j}^k\sim 1$ for all $i,j,k$ compatible with fusion rules.

\begin{lemma}
$\lambda_{1,k}^{k}\neq 0$,
$\lambda_{k,1}^{k}\neq 0$ 
,and $\lambda_{k,k}^{1}\neq 0$.
\end{lemma}
\begin{proof}
The first claim is a direct consequence from the definition that $U^k$ is an irreducible $U^1$-module.

For any $u^{k}\in U^k$, using skew symmetry of $Y\left(\cdot,z\right)$,
we have
\[
Y(u^{k},z)u^{1}=e^{zL\left(-1\right)}Y\left(u^{1},-z\right)u^{k}=\lambda_{1,k}^{k}\cdot e^{zL\left(-1\right)}\mathcal{I}_{1,k}^{k}\left(u^{1},-z\right)u^{k}=\lambda_{k,1}^{k}\cdot\mathcal{I}_{k,1}^{k}\left(u^{k},z\right)u^{1}.
\]
So $\lambda_{k,1}^{k}\not=0$.

For $u^{k},v^{k}\in U^k$, we have
\[
\left\langle Y(u^{k},z)v^{k},u^{1}\right\rangle =\left\langle v^{k},Y\left(e^{zL\left(1\right)}\left(-z^{-2}\right)^{L\left(0\right)}u^{k},z^{-1}\right)u^{1}\right\rangle .
\]
That is,
\[
\left\langle \lambda_{k,k}^{1}\cdot\mathcal{I}_{k,k}^{1}\left(u^{k},z\right)v^{k},u^{1}\right\rangle =\left\langle v^{k},\lambda_{k,1}^{k}\cdot\mathcal{I}_{k,1}^{k}\left(e^{zL\left(1\right)}\left(-z^{-2}\right)^{L\left(0\right)}u^{k},z^{-1}\right)u^{1}\right\rangle .
\]
Applying previous claim, $\lambda_{k,1}^{k}\not=0$, and hence $\lambda_{k,k}^{1}\not=0$.
\end{proof}

\begin{theorem}\label{alternation}
$\lambda_{i,j}^k\sim\lambda_{j,i}^k\sim\lambda_{j,k}^i$ for all $i,j,k$ compatible with fusion rules.
\end{theorem}
\begin{proof}
Consider $(1,i,j,k)\sim (1,j,i,k)$.
Note that there will be only one monomial in each side of the equation
since $N_{a,b}^1\neq 0$ if and only if $a=b$.
We have
$\lambda_{j,k}^i\lambda_{i,i}^1\sim\lambda_{i,k}^j\lambda_{j,j}^1.$
So $\lambda_{j,k}^i\sim \lambda_{i,k}^j$ by previous claim.
The change of basis on each side is one-dimensional and invertible. Skew symmetry also gives $\lambda_{i,j}^k\sim\lambda_{j,i}^k$.
\end{proof}

Transpositions of indices of $\lambda_{a,b}^c$ will be omitted from now on.

If a sum $V$ of the $U^i$ contains $U^1$ and is closed under the vertex operator, it satisfies the hypotheses on the extension in Theorem~\ref{c1}. Indeed, a nonzero ideal of $V$ contains some $U^i$ by multiplicity one. Since $\lambda_{i,i}^1\ne0$, that ideal contains $U^1$ and hence the vacuum. The invariant form restricts nondegenerately to $V$, since the $U^i$ are inequivalent and self-dual. Thus Theorem~\ref{c1}(ii) gives fusion closure for the corresponding $Q^i$.



\begin{lemma}\label{3,i,i-2}
$\lambda_{i,i-2}^3\sim 1$ for all $i$ compatible with fusion rules.
\end{lemma}
\begin{proof}
We already know $\lambda_{3,1}^3\sim 1$. 
Assume $\lambda_{3,i-2}^i\sim 1$ for $i=3, 5,7,\dots,2t+1$,
but $\lambda_{3,2t+1}^{2t+3}=0$.

If we can prove that $\lambda_{p,q}^{r}=0$ for all 
$1\leq p,q\leq 2t+1,r\geq 2t+3$ when $N_{p,q}^r=1$, 
then $V=\sum_{n=0}^{t}U^{2n+1}$ will be a subVOA of $\mathcal{U}$.
But $Q^{2t+3}$ occurs in $Q^3\boxtimes Q^{2t+1}$, contrary to the fusion closure just noted.

To get the contradiction we need, 
we will firstly prove that $\lambda_{p,q}^{r}=0$ for all $r=2t+3$ and $p,q$ such that $p+q=2t+4$.
Then we will prove that $\lambda_{p,q}^r=0$ for all $r\geq 2t+3,p\leq 2t+1,q\leq 2t+1$.
Thus we can get the contradiction.

For $r=2t+3$ and $p+q=2t+4$,
we consider $(3,p,q-2,r)\sim (3,q-2,p,r)$.
And we get
$$
\left\{
\lambda_{3,p}^{p-2}\lambda_{r,q-2}^{p-2},\quad
\lambda_{3,p}^{p}\lambda_{r,q-2}^{p},\quad
\lambda_{3,p}^{p+2}\lambda_{r,q-2}^{p+2}
\right\}
\sim
\left\{
\lambda_{3,q-2}^{q-4}\lambda_{r,p}^{q-4},\quad
\lambda_{3,q-2}^{q-2}\lambda_{r,p}^{q-2},\quad
\lambda_{3,q-2}^{q}\lambda_{r,p}^{q}\right\}.
$$
Note that by fusion rules we have
$N_{r,q-2}^{p-2}=N_{r,q-2}^{p}=N_{r,p}^{q-4}=N_{r,p}^{q-2}=0$,
and that our assumption asserts $\lambda_{3,i-2}^{i}\sim 1$ for odd $i$ such that $3\leq i \leq 2t+1$.
As a result, the equation becomes
$\lambda_{r,q-2}^{p+2}
\sim\lambda_{r,p}^{q}$ for all odd $3\leq p\leq 2t-1$.
So 
$$\lambda_{r,p}^{q}
\sim\lambda_{r,p-2}^{q+2}\sim\cdots\sim \lambda_{r,3}^{2t+1}=0,$$
and the first step is finished.




Now we consider the case when $p\leq 2t-1,q\leq 2t+1,r\geq 2t+3$.
Fix $r\geq2t+3$ and argue by induction on the sum of the two input labels. An input label $1$ gives zero by the unit rule. To treat $(p+2,q)$, assume the assertion for smaller sums.

Consider $(3,p,q,r)\sim(3,q,p,r)$,
we have
$$\left\{
	\lambda_{3,p}^{p-2}\lambda_{q,r}^{p-2},\quad
\lambda_{3,p}^{p}\lambda_{q,r}^{p},\quad
\lambda_{3,p}^{p+2}\lambda_{q,r}^{p+2}\right\}
\sim
\left\{
\lambda_{3,q}^{q-2}\lambda_{p,r}^{q-2},\quad
\lambda_{3,q}^{q}\lambda_{p,r}^{q},\quad
\lambda_{3,q}^{q+2}\lambda_{p,r}^{q+2}\right\}.$$
It also holds that $\lambda_{p-2,q}^{r}=\lambda_{p,q}^{r}=\lambda_{p,q-2}^{r}=0$ by assumption.
Since $p\leq 2t-1$, $\lambda_{3,p}^{p+2}\sim 1$,
we get $\lambda_{q,r}^{p+2}\sim\lambda_{3,q}^{q+2}\lambda_{p,r}^{q+2}$.

If $q=2t+1$, then by assumption $\lambda_{q,r}^{p+2}\sim\lambda_{q,3}^{q+2}=0$.
For $q<2t+1$, the same relation replaces $(p+2,q)$ by $(p,q+2)$ at the same sum.
Iteration reaches either the input label $1$ or the boundary $q=2t+1$. Both give zero.
We get $\lambda_{p,q}^r=0$ for any $p,q\leq 2t+1$.
Thus the second step is finished and we get the contradiction we need.
\end{proof}

\begin{lemma}{\label{recursion}}
Assume for any odd $i$, $\lambda_{i,i}^3\sim 1$.
For any odd $p$, if $\lambda_{i,j}^k\sim 1$ holds as long as
$\min\left\{i,j,k\right\}\leq p$
with $i,j,k$ compatible with fusion rules,
then $\lambda_{i,j}^{p+2}\sim 1$.
\end{lemma}

\begin{proof}
$(3,p,i,j)\sim(3,i,p,j)$ implies
$$
\left\{\lambda_{3,p}^{p-2}\lambda_{i,j}^{p-2},\quad
\lambda_{3,p}^{p}\lambda_{i,j}^{p},\quad
\lambda_{3,p}^{p+2}\lambda_{i,j}^{p+2}\right\}
\sim
\left\{
\lambda_{3,i}^{i-2}\lambda_{p,j}^{i-2},\quad
\lambda_{3,i}^{i}\lambda_{p,j}^{i},\quad
\lambda_{3,i}^{i+2}\lambda_{p,j}^{i+2}\right\}.
$$
Using assumptions and the previous lemma, we get $\lambda_{i,j}^{p+2}\sim 1$.
\end{proof}

Now if we can prove $\lambda_{i,i}^3\sim 1$ for all $i$ compatible with fusion rules, 
we will get $\lambda_{i,j}^k\neq 0$ for all $i,j,k$ compatible with fusion rules by induction using the previous lemma.
In order to prove this fact, 
we consider separately the two parities of the GKO level $m$.

\subsection{When $m$ is odd}
The case when $m$ is assumed to be odd is under our first consideration.


\begin{lemma}
$\lambda_{m,m}^3\sim 1$.
\end{lemma}
\begin{proof}
It can be checked directly that $Q^m\boxtimes Q^m=Q^1\oplus Q^3$.
If $\lambda_{m,m}^3=0$, then $V=U^1+U^m$ is a subVOA of $\mathcal{U}$.
Theorem~\ref{c1}(ii) would then make $\{Q^1,Q^m\}$ closed under fusion, contrary to this fusion rule.
\end{proof}

\begin{lemma}
$\lambda_{a,b}^m\sim 1$ for all $a,b$ compatible with fusion rules.
\end{lemma}
\begin{proof}
By $(m,i,m,i-2)\sim(m,m,i,i-2)$ we have
$$\lambda_{i,m-i+3}^m\lambda_{m,i-2}^{m-i+3}\sim\lambda_{m,m}^3\lambda_{i,i-2}^3\sim 1,$$
then
$\lambda_{m,i}^{m-i+3}\sim \lambda_{m,i-2}^{m-i+3}\sim 1$
and these are all cases of $a,b$ that is compatible with fusion rules.
\end{proof}
\begin{lemma}
$\lambda_{a,b}^3\sim 1$ for all $a,b$ compatible with fusion rules.
\end{lemma}
\begin{proof}
Consider $(m,m,i,j)$ and we have
$$
\left\{\lambda_{m,m}^1\lambda_{i,j}^1,\quad
\lambda_{m,m}^3\lambda_{i,j}^3\right\}
\sim
\left\{\lambda_{m,i}^{m+1-i}\lambda_{m,j}^{m+1-i},\quad
\lambda_{m,i}^{m+3-i}\lambda_{m,j}^{m+3-i}\right\}.$$
Applying previous lemmas, the proof is direct.
\end{proof}



\begin{theorem}
If $m$ is odd, then $\lambda_{i,j}^k\neq 0$ for every $N_{i,j}^k =1$.
\end{theorem}
\begin{proof}
This is a direct consequence of previous lemma and Lemma (\ref{recursion}) with proper
alternations of indices.
\end{proof}

\subsection{When $m$ is even}
Thanks to Lemma (\ref{recursion}), 
we only need
$\lambda_{i,i}^3\sim 1$
whenever $N_{i,i}^3=1$.
Several lemmas are needed to establish the proof
since $N_{m+1,m+1}^3=0$.

\begin{lemma}\label{5,i,i}
$\lambda_{i,i}^5\sim 1$, $\forall i$ when $N_{i,i}^5=1$.
\end{lemma}
\begin{proof}
Consider $(3,3,i,i)$, $i\geq 3$.
We get 
$$
\left\{1,\quad
\lambda_{3,3}^3\lambda_{i,i}^3,\quad
\lambda_{3,3}^5\lambda_{i,i}^5\right\}
\sim 
\left\{1,\quad
\left(\lambda_{i,i}^3\right)^2,\quad
1\right\}.$$
Thus $\lambda_{i,i}^5\sim 1$.
\end{proof}

\begin{lemma}\label{5,i,i+4}
$\lambda_{i,i+4}^5\sim 1$, $\forall i$ when $N_{i,i+4}^5=1$.
\end{lemma}
\begin{proof}
This comes from $(3,3,i,i+4)$:
$$
\left\{\lambda_{3,3}^1\lambda_{i,i+4}^1,\quad
\lambda_{3,3}^3\lambda_{i,i+4}^3,\quad
\lambda_{3,3}^5\lambda_{i,i+4}^5\right\}
\sim
\left\{
\lambda_{3,i}^{i-2}\lambda_{3,i+4}^{i-2},\quad
\lambda_{3,i}^{i}\lambda_{3,i+4}^{i},\quad
\lambda_{3,i}^{i+2}\lambda_{3,i+4}^{i+2}\right\}.
$$
Since $N_{i,i+4}^1=N_{i,i+4}^3=N_{3,i+4}^{i-2}=N_{3,i+4}^{i}=0$ by fusion rules,
and we proved Lemma (\ref{3,i,i-2}),
the equation becomes $\lambda_{i,i+4}^5\sim 1$.
\end{proof}


\begin{lemma}
If either $\lambda_{3,3}^3=0$ or $\lambda_{5,5}^3=0$,
then $\lambda_{4i+1,4i+1}^3=0$ and $\lambda_{5,i+2}^i=0$, $\forall i$.
\end{lemma}
\begin{proof}
Consider $(3,3,i,i+2)$ and $(3,5,i,i+2)$,
and we get 
\begin{equation}\label{(3,3,i,i+2)}
\begin{aligned}
	\left\{
\lambda_{3,3}^3,\quad\lambda_{i,i+2}^{5}\right\}&\sim
\left\{\lambda_{3,i}^{i},\quad
\lambda_{3,i+2}^{i+2}\right\},\\
\left\{\lambda_{3,5}^5\lambda_{i,i+2}^{5},\quad
\lambda_{i,i+2}^{7}\right\}
&\sim
\left\{\lambda_{3,i}^{i}\lambda_{5,i+2}^{i},\quad
1\right\}.
\end{aligned}
\end{equation}
The first line holds for odd $3\leq i\leq m-1$, the second for odd $3\leq i\leq m-3$. At $i=m-1$, the two expansions of $(3,5,m-1,m+1)$ have just one term each, with intermediate labels $3$ and $m-3$, respectively. The first line at $i=m-1$ gives $\lambda_{3,3}^3\sim\lambda_{3,m-1}^{m-1}$.


The case $\lambda_{5,5}^3=0$:
The second line gives
$\lambda_{3,i}^{i}\lambda_{5,i+2}^{i}=0$.
Thus, if $\lambda_{3,i}^i\ne0$, then $\lambda_{5,i+2}^i=0$. Counting nonzero terms in the first line now gives the following two cases.
If $\lambda_{3,3}^3=0$, 
we get $\lambda_{3,i}^{i}=0$ for all $i$,
thus $\lambda_{5,i+2}^i=0$, $\forall i$.
If $\lambda_{3,3}^3\sim 1$, 
we have $\lambda_{4t-1,4t-1}^3\sim 1$ and $\lambda_{4t+1,4t+1}^3=0$
for any $t\geq 1$.
Furthermore, $\lambda_{5,i+2}^i=0$, $\forall i$.

The case $\lambda_{3,3}^3=0$:
Put $i=3$ in the second equation and we get $\lambda_{5,5}^3=0$.
It has been discussed already.
\end{proof}


\begin{lemma}
$\lambda_{3,3}^3\sim\lambda_{5,5}^3\sim 1$.
\end{lemma}
\begin{proof}
Otherwise if we can proof $\lambda_{4p+1,4q+1}^{4r+3}=0$ for all $p,q\geq 1$ and $r\geq 0$,
we will get a subVOA $V$ of $\mathcal{U}$ where $V=\sum_{0\leq 4n\leq m} U^{4n+1}$.
This contradicts Theorem~\ref{c1}(ii), since $Q^3$ occurs in $Q^5\boxtimes Q^5$ but does not occur in $V$ as a second factor.


It's enough to assume $\lambda_{5,5}^3\sim 0$ and use previous lemma to assert the contradiction.
To see the closure also at the endpoints, note that every mode of $U^5$ preserves $V$. Only the shifts by $2$ or $-2$ can change the label modulo $4$, and the previous lemma makes both components zero. The modes of $U^1$ preserve $V$ as well. Conversely, Lemma~\ref{5,i,i+4} gives every successive summand $U^{4n+5}$ from $U^{4n+1}$ and $U^5$ whenever $4n+5\leq m+1$. By irreducibility and multiplicity one, $V$ is exactly the subVOA generated by $U^1$ and $U^5$.







Our origion assumption asserts the case $r'=0,p'=1,q'=1$.
We need to show that if $\lambda_{4p+1,4q+1}^{4r+3}=0$ for every $p\leq p',q\leq q',r\leq r'$,
then $\lambda_{4p'+1,4q'+1}^{4(r'+1)+3}=0$ and  $\lambda_{4(p'+1)+1,4q'+1}^{4r'+3}=0$.

Consider $(5,4r+3,4p+1,4q+1)$,
we get
$$\left\{\begin{aligned}
&\lambda_{5,4r+3}^{4r-1}\lambda_{4p+1,4q+1}^{4(r-1)+3},\quad
\lambda_{5,4r+3}^{4r+1}\lambda_{4p+1,4q+1}^{4r+1},\quad
\lambda_{5,4r+3}^{4r+3}\lambda_{4p+1,4q+1}^{4r+3}\\
&\lambda_{5,4r+3}^{4r+5}\lambda_{4p+1,4q+1}^{4(r+1)+1},\quad
\lambda_{5,4r+3}^{4r+7}\lambda_{4p+1,4q+1}^{4(r+1)+3}
\end{aligned}\right\}
$$
$$
\sim\left\{\begin{aligned}
 & \lambda_{5,4p+1}^{4p-3}\lambda_{4r+3,4q+1}^{4(p-1)+1},\quad
\lambda_{5,4p+1}^{4p-1}\lambda_{4r+3,4q+1}^{4(p-1)+3},\quad
\lambda_{5,4p+1}^{4p+1}\lambda_{4r+3,4q+1}^{4p+1}\\
&\lambda_{5,4p+1}^{4p+3}\lambda_{4r+3,4q+1}^{4p+3},\quad
\lambda_{5,4p+1}^{4p+5}\lambda_{4r+3,4q+1}^{4(p+1)+1}
\end{aligned}
\right\}.
$$

Since we have Lemma (\ref{5,i,i}), Lemma (\ref{5,i,i+4}) and assumption against previous Lemma,
we already know $\lambda_{5,i}^j$ for all $i,j$.
For admissible shifts in both expansions this simplifies to
$$\left\{\lambda_{4p+1,4q+1}^{4(r-1)+3},\quad
\lambda_{4p+1,4q+1}^{4(r+1)+3}
\right\}
\sim 
\left\{\lambda_{4r+3,4q+1}^{4(p-1)+1},\quad
\lambda_{4r+3,4q+1}^{4(p+1)+1}\right\}.$$

By assumption we get $\lambda_{4(p+1)+1,4q+1}^{4r+3}\sim\lambda_{4p+1,4q+1}^{4(r+1)+3}$.
So we have
$$\lambda_{4(p+1)+1,4q+1}^{4r+3}\sim\lambda_{4p+1,4q+1}^{4(r+1)+3}\sim\cdots\sim\lambda_{5,4q+1}^{4p+4r+3}\sim 0,$$
the last relation coming from the previous lemma that
$\lambda_{5,i+2}^i=0$
and from fusion rules that $N_{i,j}^5=0$ if $|i-j|>4$.

Put $r=r'$ and we get what we need.
\end{proof}

\begin{lemma}
$\lambda_{i,i}^3\neq 0$ for any odd $i$ compatible with fusion rules.
\end{lemma}
\begin{proof}
We already have $\lambda_{3,3}^3\sim\lambda_{5,5}^3\sim 1$.
Assume $\lambda_{k,k}^3\sim 1$ for $k\leq t$, where $t\geq 5$, but $\lambda_{t+2,t+2}^3=0$.

Recall the first equation in (\ref{(3,3,i,i+2)}).
Put $i=t$, and we get $\lambda_{t,t+2}^5=0$.
Put $i=t+2$, and we get $\lambda_{t+2,t+4}^5=0$.

Consider $(3,5,i,i)$, $3\leq i\leq m-1\ \text{odd}$,
$$\left\{1,\quad
\lambda_{i,i}^7\right\}
\sim
\left\{\lambda_{5,i}^{i-2},\quad
\lambda_{5,i}^{i+2}\right\}.
$$   

Put $i=t+2$ and we have
$\{1,\lambda_{t+2,t+2}^7\}\sim\{0,0\}.$
We get the contradiction by Theorem (\ref{samerank}).
\end{proof}

\begin{theorem}
If $m$ is even, then $\lambda_{i,j}^k\neq 0$ for every $N_{i,j}^k =1$.
\end{theorem}
\begin{proof}
The previous lemma asserts Lemma (\ref{recursion}). 
We can get what we want by recursion.
\end{proof}

\end{document}